\documentclass[letterpaper,11pt]{article}
\usepackage{amssymb, amsmath, amsthm}
\usepackage{amsfonts}
\numberwithin{equation}{section}
\usepackage[colorlinks=true,linkcolor=blue,anchorcolor=blue,citecolor=blue,urlcolor=blue,plainpages=false,pdfpagelabels]{hyperref}
\usepackage{color}
\usepackage{fullpage}
\usepackage{graphicx}
\usepackage{enumerate}
\usepackage{mathrsfs}
\usepackage{mathdots} 
\usepackage[textsize=scriptsize]{todonotes}
\usepackage{xifthen}
\usepackage{booktabs}
\usepackage[margin=30pt,font=small,labelfont=bf,skip=0pt]{caption}
\usepackage[capitalise]{cleveref}
\crefname{equation}{equation}{equations}
\crefname{figure}{figure}{figures}
\crefname{section}{Section}{Sections}
\crefname{subsection}{Subsection}{Subsections}

\usepackage{enumitem}
\setlist[enumerate]{leftmargin=.5in}
\setlist[itemize]{leftmargin=.5in}

\usepackage{algorithm}
\usepackage[noend]{algpseudocode}
\usepackage{comment}
\usepackage[numbers,sort,compress]{natbib}
\usepackage{xspace}

\usepackage{tikz}
\usepackage{tikz-cd}
\usetikzlibrary{arrows}
\tikzset{>=latex}
\tikzset{node distance=50pt, auto}
\usetikzlibrary{decorations.pathreplacing,angles,quotes}
\usepackage[all]{xy}
\usepackage{multirow,booktabs}

\newtheorem{theorem}{Theorem}[section]

\newtheorem{proposition}[theorem]{Proposition}

\theoremstyle{definition}
\newtheorem{definition}[theorem]{Definition}
\theoremstyle{remark}
\newtheorem{remark}[theorem]{Remark}

\newcommand{\bounds}[2]{\{#1,\ldots,#2\}}

\newcommand\Fa{X}
\newcommand\Fb{Y}
\newcommand\Fc{Z}
\newcommand{\sgr}[1]{\mathrm{gr}\ifthenelse{\isempty{#1}}{}{(#1)}}
\newcommand{\gr}[1]{\mathrm{bgr}\ifthenelse{\isempty{#1}}{}{(#1)}}
                                                                                                  
\newcommand\K{K}
\newcommand{\eb}{\mathbf e}
\newcommand{\ind}{x}
\newcommand\ma{f}
\newcommand\mb{g}

\renewcommand{\c}{\mathbf{y}}
\renewcommand{\d}{\mathbf{z}}

\newcommand\StdRed{\mathsf{GrRed}}

\newcommand\StdRedSub{\mathsf{BiRedSub}}
\newcommand\StdRedSubAux{\mathsf{BiRedSubAux}}

\newcommand\KerBetti{\mathsf{KerBetti}}
\newcommand\KerBasis{\mathsf{KerBasis}}
\newcommand\KerSchr{\mathsf{KerSchr}}
\newcommand\MinGens{\mathsf{MinGens}}
\newcommand\MinimizePres{\mathsf{MinimizePres}}

\newcommand{\Z}{\mathbb{Z}}		
\newcommand{\R}{\mathbb{R}}		

\newcommand{\g}{\mathbf g}

\newcommand{\nll}{\mathsf{null}}
\newcommand{\nr}{m}
\newcommand{\nc}{n}

\newcommand{\Rc}[0]{B}

\renewcommand{\S}[0]{B'}

\newcommand{\X}[0]{\mathcal F}
\newcommand{\y}{\mathbf{y}}
\newcommand{\z}[0]{\mathbf z}

\newcommand{\lows}{{\sf pivs}}

\newcommand{\firep}{\text{FI-Rep}\xspace}

\DeclareMathOperator{\im}{im}

\DeclareMathOperator{\coker}{coker}

\DeclareMathOperator{\rank}{rank}

\DeclareMathOperator{\id}{Id}

\DeclareMathOperator{\dimf}{hf}
\DeclareMathOperator{\nullityf}{nullf}
\DeclareMathOperator{\rankf}{rankf}

\DeclareMathOperator{\Id}{Id}

\DeclareMathOperator{\nullity}{nullity}
\DeclareMathOperator{\supp}{supp}

\DeclareMathOperator{\pivot}{\rho}
\DeclareMathOperator{\grid}{grid}

\usetikzlibrary{patterns}

\pgfdeclarepatternformonly{thickstripes_pos}{\pgfpoint{0cm}{0cm}}{\pgfpoint{1cm}{1cm}}{\pgfpoint{1cm}{1cm}}
{
  \foreach \i in {0.25cm, 0.75cm}
  {
    \pgfpathmoveto{\pgfpoint{\i}{0cm}}
    \pgfpathlineto{\pgfpoint{1cm}{1cm - \i}}
    \pgfpathlineto{\pgfpoint{1cm}{1.25cm - \i}}
    \pgfpathlineto{\pgfpoint{\i - 0.25cm}{0cm}}
    \pgfpathclose%
    \pgfusepath{fill}
    \pgfpathmoveto{\pgfpoint{0cm}{\i}}
    \pgfpathlineto{\pgfpoint{1cm - \i}{1cm}}
    \pgfpathlineto{\pgfpoint{0.75cm - \i}{1cm}}
    \pgfpathlineto{\pgfpoint{0cm}{\i + 0.25cm}}
    \pgfpathclose%
    \pgfusepath{fill}
  }
}
\pgfdeclarepatternformonly{thickstripes_neg}{\pgfpoint{0cm}{0cm}}{\pgfpoint{1cm}{1cm}}{\pgfpoint{1cm}{1cm}}
{
  \foreach \i in {0.25cm, 0.75cm}
  {
    \pgfpathmoveto{\pgfpoint{1cm - \i}{0cm}}
    \pgfpathlineto{\pgfpoint{0cm}{1cm - \i}}
    \pgfpathlineto{\pgfpoint{0cm}{1.25cm - \i}}
    \pgfpathlineto{\pgfpoint{1.25cm - \i}{0cm}}
    \pgfpathclose%
    \pgfusepath{fill}
    \pgfpathmoveto{\pgfpoint{1cm}{\i}}
    \pgfpathlineto{\pgfpoint{\i}{1cm}}
    \pgfpathlineto{\pgfpoint{\i + 0.25cm}{1cm}}
    \pgfpathlineto{\pgfpoint{1cm}{\i + 0.25cm}}
    \pgfpathclose%
    \pgfusepath{fill}
  }
}

\begin{document}

\title{Computing Minimal Presentations and Bigraded Betti Numbers of\\ 2-Parameter Persistent Homology}

\author{Michael Lesnick\thanks{SUNY Albany, Albany, NY, USA; \texttt{mlesnick@albany.edu}} \and Matthew Wright\thanks{St.\ Olaf College, Northfield, MN, USA; \texttt{wright5@stolaf.edu}}}

\maketitle


\begin{abstract}
Motivated by applications to topological data analysis, we give an efficient algorithm for computing a (minimal) presentation of a bigraded $K[x,y]$-module $M$, where $K$ is a field.  The algorithm takes as input a short chain complex of free modules $X\xrightarrow{f} Y \xrightarrow{g} Z$ such that $M\cong \ker{g}/\operatorname{im}{f}$.  It runs in time 
$O(|X|^3+|Y|^3+|Z|^3)$ and requires $O(|X|^2+|Y|^2+|Z|^2)$ memory, where $|\cdot |$ denotes the rank.  Given the presentation computed by our algorithm, the bigraded Betti numbers of $M$ are readily computed.  Our approach is based on a simple matrix reduction algorithm, slight variants of which compute kernels of morphisms between free modules, minimal generating sets, and Gr\"obner bases.  Our algorithm for computing minimal presentations has been implemented in RIVET, a software tool for the visualization and analysis of two-parameter persistent homology.  In experiments on topological data analysis problems, our implementation outperforms the standard computational commutative algebra packages Singular and Macaulay2 by a wide margin.
\end{abstract}

\section{Introduction}

\subsection{Persistence Modules, Minimal Presentations, and Betti Numbers}\label{Sec:Persistence_Modules}

Let $\K$ be a field.  For $d\in \{1,2,\ldots\}$, a \emph{($d$-parameter) persistence module} is defined to be a $\K[\ind_1,\ldots,\ind_d]$-module $M$ equipped with a \emph{$d$-grading}. A $d$-grading of $M$ is a vector space decomposition $M=\oplus_{\z\in \Z^d} M_\z$ such that $\ind^i M_\z\subset M_{\z+\eb_i}$ for all $\z\in \Z^d$ and $i\in \{1,\ldots,d\}$, where $\eb_i$ denotes the $i^{\mathrm{th}}$ standard basis vector.  We refer to 2-parameter persistence modules as \emph{bipersistence modules.} Persistence modules are standard objects of study in commutative algebra \cite{eisenbud2005geometry,miller2005combinatorial}.  

In topological data analysis (TDA) \cite{carlsson2009topology,oudot2015persistence,edelsbrunner2010computational,wasserman2018topological,boissonnat2018geometric}, $d$-parameter persistence modules arise as invariants of data, in the context of \emph{multi-parameter persistent homology}.  To explain, let us define a \emph{($d$-parameter) filtration} to be a collection of simplicial complexes $\{\X_\z\}_{\z \in \Z^d}$ such that $\X_\z\subset \X_{\z+\eb_i}$ for all $\z\in \Z^d$ and $i\in \{1,\ldots,d\}$.   A number of well-known constructions in TDA associate a filtration to data, with the aim of topologically encoding information about the shape of the data.  Applying the $i^{\mathrm{th}}$ homology functor with coefficients in $\K$ to the filtration, one obtains a $d$-parameter persistence module, which serves an algebraic descriptor of the data.  The $d=1$ case has received the most attention, but it is often quite natural to consider $d$-parameter filtrations for $d\geq 2$, and this is currently a very active area of research.  The  $d=2$ case is of particular interest \cite{carlsson2009theory,lesnick2015interactive,miller2020homological,botnan2018algebraic,cochoy2016decomposition}, in part because 2-parameter filtrations arise in the study of point cloud data with noise or non-uniform density \cite{carlsson2009theory, blumberg2020stability,sheehy2012multicover,edelsbrunner2018multi,corbet2021computing}.

The isomorphism type of a finitely generated $1$-parameter persistence module is specified by a collection of pairs $(a,b)$ with $a<b\in \Z\cup\{\infty\}$, called a \emph{barcode}.  When $d\geq 2$, the representation theory of $d$-parameter persistence modules is known to be wild, and there is no simple invariant which completely encodes the isomorphism type of a module \cite{carlsson2009theory}.  Nevertheless, for the purposes of TDA, one can consider incomplete invariants of a persistence module as surrogates for the barcode, and a number of ideas for this have been proposed, e.g., see \cite{carlsson2009theory,cerri2013betti,lesnick2015interactive,vipond2018multiparameter,harrington2017stratifying}.

Although no simple, complete invariant of a $d$-parameter persistence module is available, one can specify the isomorphism type of a finitely generated persistence module via a \emph{minimal presentation}.  Concretely, this is a matrix with field coefficients, with each row and each column labeled by an element of $\Z^d$.  Because minimal presentations are not unique, they cannot be directly used in the way barcodes are used in TDA, e.g., as input to machine learning algorithms or statistical tests.  However, they serve as useful computational intermediates.

The \emph{(multi-graded) Betti numbers} are standard invariants of a persistence module, and play an important role in parts of commutative algebra \cite{miller2005combinatorial,eisenbud2005geometry}.  For a finitely generated $d$-parameter persistence module $M$ and $j\in \{0,\ldots,d\}$, the $j^{\mathrm{th}}$ (multi-graded) Betti number of $M$ at grade $\z$, denoted $\beta_j^M(\z)$, is the number of elements at grade $\z$ in a basis for the $j^{\mathrm{th}}$ module in a free resolution for $M$; see \cref{Sec:Preliminaries}.  The $\beta_j^M(\z)$ thus define a function $\beta_j^M:\Z^d\to \mathbb{N}$.  
When $d=2$, we refer to the $\beta_j^M(\z)$ as \emph{bigraded Betti numbers}.  

For persistence modules arising in TDA, the multi-graded Betti numbers offer interesting partial information about the coarse-scale geometry of the data \cite{carlsson2009theory}.  The bigraded Betti numbers of a bipersistence module are readily visualized as a collection of colored dots in the plane \cite{lesnick2015interactive}.
They also have an another application to TDA: In recent work, the authors have introduced a software tool called RIVET for the interactive visualization of bipersistence modules, designed with needs of TDA in mind \cite{lesnick2015interactive,rivet}.  The tool provides a visualization of the bigraded Betti numbers of a bipersistence module $M$, as well as visualizations of two other invariants $M$, the \emph{Hilbert function} and \emph{fibered barcode}.  The fibered barcode is a collection of barcodes of certain 1-parameter restrictions of the bipersistence module.  A central feature of RIVET is a framework for interactive visualization of the fibered barcode.  RIVET's interactivity makes use of a novel data structure called the \emph{augmented arrangement} of $M$, which is a line arrangement in the plane, together with additional data of a barcode $\mathcal B_f$ at each face $f$ of the arrangement.  The definition of line arrangement is given in terms of $\beta_0^M$ and $\beta_1^M$, and the algorithm for computing it involves the computation of $\beta_0^M$ and $\beta_1^M$ as a subroutine.  Our experience suggests that computing the barcodes $\mathcal B_f$ is best done by first computing a minimal presentation for $M$. 

\subsection{Our Contributions}
Motivated by TDA applications, and in particular by RIVET's data analysis pipeline, this paper considers the problems of computing a (minimal) presentation and the bigraded Betti numbers of a bipersistence module $M$.  We assume that $M$ is given implicitly: We take the input to be a chain complex of free bipersistence modules
\begin{equation}\label{Eq:Short_Chain_Complex}
\Fa\xrightarrow{\ma} \Fb \xrightarrow{\mb} \Fc
\end{equation}
 with $M\cong \ker{\mb}/\im{\ma}$.  
We provide algorithms for both problems requiring $O(|X|^3+|Y|^3+|Z|^3)$ time and $O(|X|^2+|Y|^2+|Z|^2)$ memory, where for a free module $F$, $|F|$ denotes the rank of $F$, i.e., the size of any basis for $F$.  Here and throughout, when stating complexity bounds, we assume for simplicity that an elementary arithmetic operation in the field $K$ requires $O(1)$ time, and that storing an element of $K$ requires $O(1)$ memory.  Virtually all TDA computations are done with $K$ a finite field, where these assumptions hold.

\paragraph{Chain Complexes of Free Modules from Filtrations}
We next briefly explain how the short chain complexes of free modules \cref{Eq:Short_Chain_Complex} arise in TDA.  Given a $d$-parameter filtration $\X$, one has an associated chain complex of $d$-parameter persistence modules 
\[\cdots \xrightarrow{\partial_{i+1}} C_i(\X)\xrightarrow{\partial_{i}} C_{i-1}(\X)\xrightarrow{\partial_{i-1}}\cdots \xrightarrow{\partial_{1}} C_{0}(\X)\to 0,\]
where $C_i(\X)_\z:=C_i(\X_\z;\K)$ is the usual simplicial chain vector space with coefficients in $\K$, and the internal maps in $C_i(\X)$ are inclusions.  Let $H_i(\X)$ denote the $i^{\mathrm{th}}$ homology module of this chain complex.

Often in TDA, $\X$ is defined in a way that ensures each $C_i(\X)$ is free.  In this case we say $\X$ is \emph{1-critical}; otherwise, we say $\X$ is \emph{multi-critical}.  For example, the function-Rips bifiltration \cite{carlsson2009theory,blumberg2020stability} and the rhomboid bifiltration \cite{corbet2021computing,edelsbrunner2018multi} are 1-critical, while the degree bifiltrations \cite{lesnick2015interactive,blumberg2020stability} and the Rips trifiltration for time-varying data defined in \cite{kim2020spatiotemporal} are multi-critical.   If $\X$ is multi-critical, a simple construction proposed by Chacholski, Scolamiero, and Vaccarino \cite{chacholski2017combinatorial} takes as input the short chain complex \[C_{i+1}(\X)\xrightarrow{\partial_{i+1}} C_i(\X)\xrightarrow{\partial_{i}}C_{i-1}(\X)\] and yields a chain complex of free $d$-parameter persistence modules \[X\xrightarrow{f} Y\xrightarrow{g} Z\] with $H_i(\X)\cong \ker g/\im f$.  In the 2-parameter case, this has been implemented in RIVET by Roy Zhao.

Whether the filtration is 1-critical or multi-critical, the cost of computing the filtration and its associated chain complex is an important consideration.  For some choices of filtration, such as the rhomboid bifiltration, devising a practical algorithm for this is non-trivial \cite{edelsbrunner2020simple}, while for other filtration types, such as the density-Rips bifiltrations  considered in the computational experiments of this paper, a naive approach suffices for practical use.  While the efficient computation of specific types of multi-parameter filtrations is an important topic, we do not study this here; we assume that the chain complex of free bipersistence modules (\ref{Eq:Short_Chain_Complex}) has already been computed and is available as input to our algorithms.

\paragraph{Classical Approaches to Presentation and Betti Number Computation}
Questions of computational efficiency aside, the problems of computing a minimal presentation and bigraded Betti numbers of a bipersistence module can also be solved by standard Gr\"{o}bner basis techniques, which work in much greater generality.  These techniques can in fact compute resolutions and Betti numbers of graded $\K[\ind_1,\ldots,\ind_d]$ modules, for any $d$.  Typically, such approaches rely on Schreyer's algorithm, a standard algorithm for computing kernels of homomorphisms of free modules  \cite{schreyer1980berechnung,cox1998using}, or on refinements of this algorithm \cite{erocal2016refined,la1998strategies}. 
Several variants of the Gr\"{o}bner basis approach to computing (minimal) resolutions and Betti numbers are implemented in popular computational algebra software packages such as Macaulay2, Magma, Singular, and CoCoA \cite{eisenbud2001computations,greuel2012singular,MR1484478,CoCoA}.  

But to the best of our knowledge, no prior work has focused on the particular problem of computing presentations and Betti numbers of bipersistence modules.  In this work, we introduce simpler and more efficient algorithms for this special case.

\paragraph{The Bigraded Reduction}
The core computational engine underlying the algorithms of this paper (and our main technical contribution) is a simple matrix reduction algorithm, which we call the \emph{bigraded reduction}.  This is similar to the standard matrix algorithm for computing persistent homology via left-to-right column additions \cite{zomorodian2005computing}.  The key difference is that instead of reducing an entire matrix in one pass, the bigraded reduction proceeds by reducing various submatrices of increasing size. 

Slight variants of the bigraded reduction solve three basic problems involving free bipersistence modules.  We'll call the three variants we consider V1--V3.  To explain these, let $\gamma$ be a morphism of free bipersistence modules represented by an $m\times n$ matrix; see \cref{Sec:Preliminaries} for the definitions.  As explained in \cite{chacholski2017combinatorial}, $\ker \gamma$ is free.

\begin{itemize}
\item V1  computes a basis of $\ker \gamma$ in $O(\nc(\nr+\nc)\min(\nr,\nc))$ time and $O(\nr\nc+\nc^2)$ memory.  
\item V2  computes a minimal set of generators of $\im \gamma$ in time $O(\nr\nc\cdot \min(\nr,\nc)+\nc^2)$  time and $O(\nr\nc)$ memory.  The memory bound is asymptotically tight, as it matches the worst-case size of the input.
\item V3  computes a minimal Gr\" obner basis $G$ of $\im \gamma$ in $O(\nr\nc\cdot \min(\nr,\nc)+\nc^2)$ time and $O(\nr\nc\cdot \min(\nr,\nc))$ memory.
By extending a construction shown to us by Alex Tchernev, we show that the total number of terms in $G$ is $\Theta(\nr\nc\cdot \min(\nr,\nc))$ in the worst case (\cref{Prop:Tchernev}), so the memory bound is again asymptotically tight. 
\end{itemize}

\paragraph{Comparison with Classical Approaches}
Instead of using V1, one could compute a set of generators for $\ker \gamma$ via Schreyer's algorithm, as outlined, e.g., in \cite[Chapter 5, Proposition 3.8]{cox1998using} or \cite[Chapter 15.5]{eisenbud1995commutative}, and then minimize this set of generators.  In fact, one can view V1 as a simplified and optimized variant of this approach.  That said, compared to the standard version of Schreyer's algorithm, V1 is simpler and more efficient.  In particular, we will see in \cref{Sec:Comparison} that V1 requires asymptotically less memory than a naive implementation of Schreyer's algorithm.  The key observation is that while V1 effectively computes a Gr\"obner basis of $\im \gamma$, it never stores the full Gr\"obner basis in memory; in contrast, Schreyer's algorithm does store a full Gr\" obner basis of $\im \gamma$.

Along similar lines, it is well known that one can minimize a set of generators for a submodule of a free $\K[\ind_1,\ldots,\ind_d]$-module via a Gr\" obner basis computation \cite[Section 2]{bayer1992can}.  Our minimization algorithm V2 is similar in spirit, but as above, it uses asymptotically less memory because it never stores a full Gr\" obner basis.

\paragraph{Computing Minimal Presentations and Betti Numbers via Bigraded Reduction}
The main algorithms of this paper do not explicitly use V3, i.e., they do not explicitly compute Gr\" obner bases of images, but they rely in an essential way on V1 and V2: Given a chain complex of free bipersistence modules  \[\Fa\xrightarrow{\ma} \Fb \xrightarrow{\mb} \Fc,\] computing a (non-minimal) presentation of $\ker{\mb}/\im{\ma}$ amounts to computing a basis for $\ker{\mb}$ using V1, and then using standard linear algebra to express a set of generators of $\im{\ma}$ in terms of this basis.  (Indeed, for the latter step, standard linear algebra is sufficient because the basis for $\ker{\mb}$ that we compute is also a Gr\"obner basis; see \cref{Rem:Basis_Is_a_Grobner_Basis}.)  To obtain a minimal presentation, we first use V2 to minimize the set of generators for 
$\im{\ma}$; this yields what we call a \emph{semi-minimal} presentation.  From this, we then compute a minimal presentation, using a variant of a standard matrix reduction procedure for minimizing resolutions.  The details are given in \cref{Sec:Computing_A_Presentation}. 

The problems of computing a presentation and computing the bigraded Betti numbers are closely related.  As we explain in \cref{Sec:Betti_from_Pres}, once we have computed a semi-minimal presentation of $\ker{\mb}/\im{\ma}$ as outlined above, we can obtain the Betti numbers with little additional work; fully minimizing the presentation is not necessary.  Our approach to computing a minimal presentation extends readily to the computation of a minimal resolution, via one additional kernel computation.  But to compute the Betti numbers from a presentation, we in fact do not need to compute a full resolution.

Problems arising in TDA typically are very large but very sparse; as such, we formulate our algorithms using sparse linear algebra.  It turns out that the sparse linear algebra techniques commonly employed in computations of 1-parameter persistent homology \cite{zomorodian2005computing,bauer2017phat} extend readily to the setting of this paper.

The first version of this paper also described a second approach to computing the bigraded Betti
numbers, which avoids explicit computation of a presentation, and instead relies on the well-known
Koszul homology formulae for the Betti numbers. The approach based on presentation computation
is more efficient, far simpler, and closer to standard approaches. Thus, we have chosen to omit the
Koszul homology algorithm from this version of the paper.

\paragraph{Implementation and Experiments}
RIVET implements our algorithm for computing (minimal) presentations and uses this to compute bigraded Betti numbers.  
Computational experiments, reported in \cref{Sec:Experiments}, indicate that on typical TDA input, our approach performs well enough for practical use (see also \cite{vipond2021multiparameter,keller2018persistent,wright2020topological,schiff2020characterizing} for applications to real world data), and far better than the built-in functions of the standard computational commutative algebra software packages Macaulay2 and Singular.  However, it should be noted that Macaulay2 and Singular were not designed with these kinds of problems in mind, and have not been optimized for them.  On the other hand, RIVET's implementation was done primarily with applications to data visualization in mind, and hence is also not fully optimized for our computational experiments.   Thus, our experiments speak primarily to what is possible in practice with the implementations considered, rather than to the essential algorithmic virtue of the different approaches.

Our implementation can currently handle chain complexes arising in TDA with tens of millions of generators on a desktop computer with 64GB RAM.  State-of-the art codes for computing 1-parameter persistent homology, such as Ripser \cite{bauer2017ripser,bauer2019ripser} or Eirene \cite{henselman2016matroid,hylton2017performance}, can handle far larger inputs, thanks to a number of key optimizations such as clearing \cite{chen2011persistent,bauer2017phat}, the use of cohomology \cite{de2011dualities}, and implicit representation of chain complexes \cite{bauer2019ripser}.  We expect that such optimizations can be adapted to our setting, allowing our approach to handle substantially larger input.  

\paragraph{Kerber and Rolle's Improvements}
In fact, one year after we posted the first version of this paper to arXiv, Kerber and Rolle released a paper introducing an improved version of our algorithm for computing minimal presentations \cite{kerber2021fast}.  Their version of the algorithm remains quite similar to ours at a high level (and in particular, uses the bigraded reduction in the same way), but introduces several optimizations.   
Via a large set of computational experiments, Kerber and Rolle demonstrate that in practice, these optimizations lead to major improvements in speed and memory consumption.  That said, these  optimizations do not take advantage of cohomology or implicit representations of chain complexes, and we believe that there remains considerable room for further improvements, 
at least for some important classes of bifiltrations.   The work of Kerber and Rolle will not be discussed elsewhere in the present paper, except for \cref{Rem:Kerber}, which already appeared in the first arXiv version.

\subsection{Other Related Work}\label{Sec:Related_Work}
Carlsson, Zomorodian, and Singh were the first to consider computational aspects of multi-parameter persistence modules in the TDA setting \cite{carlsson2010computing}.  Their work considers the computation of Gr\"{o}bner bases of images and kernels of morphisms of free $d$-parameter persistence modules, using the classical Buchberger's algorithm and Schreyer's algorithm.  The work of Chacholski, Scolamiero, and Vaccarino \cite{chacholski2017combinatorial}, mentioned above, also explores the computational aspects of multi-parameter persistent homology, with a focus on the case where the chain modules are not necessarily free. 

Aiming in part to address some issues with the earlier work \cite{carlsson2010computing}, the Ph.D. thesis of Jacek Skryzalin \cite{skryzalin2016numeric} revisits the problem of computing Gr\"{o}bner bases of the kernels and images of homomorphisms of free $d$-parameter persistence modules.  Skryzalin outlines an algorithm for this \cite[Algorithm 5]{skryzalin2016numeric}.  The algorithm is inspired by the well-known F4 and F5 algorithms for computing Gr\"{o}bner bases via sparse linear algebra \cite{faugere1999new,faugere2002new}.  In the case of bipersistence modules, Skryzalin's algorithm reduces to an algorithm similar to our bigraded reduction, with the same asymptotic   complexity, though his exposition is rather different and some details are different.  (Our work and Skryzalin's were done independently.)  Skryzalin does not consider the computation of presentations or Betti numbers.

Papers by Allili et al.\ and Scaramuccia et al.\ \cite{scaramuccia2018computing,allili2017reducing} have  introduced algorithms which use discrete Morse theory to simplify a multi-filtration without altering its topological structure.  Fugacci and Kerber have recently developed an algorithm that efficiently minimizes a sparse chain complex of free persistence modules \cite{fugacci2018chunk}.  This can be viewed as a purely algebraic analogue of the algorithms of \cite{scaramuccia2018computing,allili2017reducing}.
The algorithm of Fugacci and Kerber specializes to an algorithm for minimizing a semi-minimal presentation, and this is relevant to our work; see \cref{Rem:Kerber}.  

Another line of related work concerns the computation of metrics between $d$-parameter persistence modules \cite{bjerkevik2018computing,biasotti2011new,bjerkevik2017computational,dey2018computing,kerber2018exact,kerber2020efficient,bjerkevik2021asymptotic}.  This is one potentially significant application of minimal presentation computation in TDA.

\subsection{Outline}
\cref{Sec:Preliminaries} introduces basic definitions and standard results used in the paper.  
In particular, \cref{Sec:MatRed} introduces the matrix reduction used in standard persistent homology computations; this matrix reduction serves as a primitive upon which the main algorithms of this paper build.  \cref{Sec:RankComputationAtAllIndices} presents the bigraded reduction and its application to computing the kernel of a morphism of free bipersistence modules.  \cref{Sec:Comparison} compares our algorithm for computing kernels to the classical approach via Schreyer's algorithm.  \cref{Sec:Semi-Minimal_Pres} applies the ideas of \cref{Sec:RankComputationAtAllIndices} to the problem of computing a semi-minimal presentation.  \cref{Sec:Minimizing_Pres} gives an algorithm for minimizing a semi-minimal presentation. \cref{Sec:Betti_from_Pres} gives an algorithm for directly computing the bigraded Betti numbers from a semi-minimal presentation, without minimizing. 
 \cref{Sec:Experiments} reports the results of our computational experiments.   \cref{Sec:Discussion} concludes the paper with a brief discussion of directions for future work.

\section{Preliminaries}\label{Sec:Preliminaries}

\subsection{Notation and Terminology}
In what follows, let $M$ be a $d$-parameter persistence module, as defined in \cref{Sec:Persistence_Modules}.
We regard $\Z^d$ as a partially ordered set by taking \[\y=(y_1,y_2,\ldots,y_n)\leq (z_1,z_2,\ldots,z_n)=\z\] if and only if $y_i\leq z_i$ for all $i$.  For $\y\leq \z\in \Z^d$, the action of the monomial
 \[\ind^{\z-\y}:=\ind_1^{z_1-y_1}\ind_2^{z_2-y_2}\cdots\ind_2^{z_n-y_n}\] on $M$ restricts to a linear map $M_{\c,\d}:M_\c\to M_\d$.  

A \emph{morphism} $\gamma:M\to N$ of $d$-parameter persistence modules is a module homomorphism such that $\gamma(M_\d)\subset N_\d$ for all $\d\in \Z^d$.  Let $\gamma_\d$ denote the restriction of $\gamma$ to $M_\d$.  With this definition of morphism, the $d$-parameter persistence modules form an abelian category.   

We say a non-zero element $m\in M$ is \emph{homogeneous} if $m\in  M_\d$ for some $\z\in \Z^d$.  In this case, we write $\sgr{m}=\z$.  A homogeneous submodule of $M$ is one generated  by a set of homogeneous elements.  For example, given a morphism of persistence modules $\gamma:M\to N$, $\ker(\gamma)$ and $\im(\gamma)$ are homogeneous submodules of $M$ and $N$, respectively.  Homogeneous submodules are themselves persistence modules (i.e., they inherit a $d$-grading from the ambient module), as are quotients  of persistence modules by homogeneous submodules.  Henceforth, all submodules we consider will be understood to be homogeneous.

Define \[\dimf(M):\Z^d\to \mathbb{N},\] the \emph{Hilbert function of $M$}, by 
$\dimf(M)(\z)=\dim M_\z $.  Given a morphism $\gamma:M\to N$ of persistence modules, let $\rankf \gamma:=\dimf(\im \gamma)$ and $\nullityf \gamma:=\dimf(\ker \gamma)$.  We call $\rankf \gamma$ and $\nullityf \gamma$ the \emph{pointwise rank} and \emph{pointwise nullity} of $\gamma$, respectively.

\subsection{Free Persistence Modules}\label{Sec:Free_Modules}
For $d\geq 1$ and $\g\in \Z^d$, let $Q^\g$ denote the $d$-parameter persistence module given by 
\begin{align*}
Q^\g_\d &=
\begin{cases}
\K &\text{if } \g\leq \d, \\
0 &\text{otherwise,}
\end{cases}
& Q^\g_{\c,\d}=
\begin{cases}
\id_\K &{\text{if } \g\leq \c},\\
0 &{\text{otherwise.}}
\end{cases}
\end{align*}
We say a $d$-parameter persistence module $F$ is \emph{free} if there exists a multiset $\mathcal G$ of elements in $\Z^d$ such that $F\cong \oplus_{\g\in \mathcal G}\, Q^\g$.
 
Many of the standard ideas of linear algebra adapt in a straightforward way to free persistence modules.  For example, we define a \emph{basis} for a free persistence module $F$ to be  a minimal homogeneous set of generators.   Though as in linear algebra, bases are usually not unique, the number of elements at each grade in a basis for $F$ is an isomorphism invariant.  In fact, this invariant is given by the $0^{\mathrm{th}}$ bigraded Betti numbers of $F$; see \cref{Def:Min_Pres} below.

Suppose we are given an ordered basis $B$ for a 
finitely generated free persistence module $F$.  We denote the $i^{\mathrm{th}}$ element of $B$ as $B_i$.  For $\z\in \Z^d$, we can represent $v\in F_\z$ with respect to $B$ as a vector $[v]^B\in\K^{|B|}$; we take $[v]^B$ to be the unique vector such that $[v]_i^B=0$ if $\sgr{B_i}\not\leq\z$ and 
\begin{equation}\label{Eq:Linear_Combination}
v=\sum_{i: \sgr{B_i}\leq\z} [v]^B_i \ind^{\z-\sgr{B_i}}B_i.
\end{equation}
Thus, $[v]^B$ records the field coefficients in the linear combination of $B$ giving $v$.   

Along similar lines, for $B'$ an ordered basis of a free persistence module $F'$, we represent a morphism $\gamma:F\to F'$ via a matrix $[\gamma]^{\Rc,\S}$ with coefficients in the field $\K$, with each row and column labeled by an element of $\Z^d$, as follows:
\begin{itemize}
\item The $j^{\mathrm{th}}$ column of $[\gamma]^{\Rc,\S}$ is $[\gamma(B_j)]^{\S}$.
\item The label of the $j^{\mathrm{th}}$ column is $\sgr{B_j}$,
\item The label $i^{\mathrm{th}}$ row is $\sgr{B'_i}$.
\end{itemize}
It is easy to see that $[\gamma]^{\Rc,\S}$ determines $\gamma$ up to natural isomorphism.\footnote{Given morphisms of persistence modules $\gamma:M\to M'$ and $\kappa:N\to N'$, a natural isomorphism $f:\gamma\to \kappa $ is a pair of isomorphisms $f:M\to N$, $f':M'\to N'$ such that the following diagram commutes: \[
\begin{tikzcd}[ampersand replacement=\&,row sep=2.5ex]  
M\arrow["\gamma"]{r}\arrow["f",swap]{d}  \&M'\arrow["\,f'"]{d} \\
N\arrow["\kappa",swap]{r} \&N'.  
\end{tikzcd}
\]}
Where no confusion is likely, we sometimes write $[\gamma]^{\Rc,\S}$ simply as $[\gamma]$.

\subsection{Resolutions, Presentations, and Bigraded Betti Numbers}\label{Sec:BettiNumDef}
An exact sequence of free persistence modules
\[F_\bullet:=\quad \cdots \xrightarrow{\partial_3}  F_2 \xrightarrow{\partial_2} F_1 \xrightarrow{\partial_{1}} F_{0}\] 
is called a \emph{resolution} of $M$ if $\coker(\partial^1)\cong M$. 

For a persistence module $N$, let $IN\subset N$ denote the submodule generated by the images of all linear maps $N_{\y,\z}$ with $\y<\z\in \Z^d$.  We say the resolution $F_\bullet$ is \emph{minimal} if $\im \partial_i\subset IF_{i-1}$ for each $i$.  It can be shown that if $M$ is finitely generated, then a minimal resolution $F_\bullet$ of $M$ exists.  It is unique up to isomorphism, and any resolution can be obtained (up to isomorphism) from $F_\bullet$ by summing with resolutions of the form 
\[\cdots 0\to 0 \to  G \xrightarrow{\id_G} G \to 0\to 0\to \cdots \to 0\] 
where $G$ is free, and the two copies of $G$ are allowed to appear at any two consecutive indices.   For a fuller discussion of minimal resolutions, see \cite[Chapters 19 and 20]{eisenbud1995commutative} or \cite[Chapter 1]{peeva2011graded}.

A \emph{presentation} of a persistence module is a morphism $\partial:F\to F'$ of free persistence modules with $\coker(\partial)\cong M$.  Thus, a presentation for $M$ is the data of the last morphism in a free resolution for $M$.  A presentation is said to be \emph{minimal} if it extends to a minimal resolution.  Thus, minimal presentations are unique up to isomorphism and any minimal presentation can be obtained (up to isomorphism) by summing with maps of the form \[G\xrightarrow{\Id_G} G\qquad \textup{or}\qquad G\to 0,\] where $G$ is free.  

It follows from \cref{Sec:Free_Modules} that we can represent the presentation $\partial:F\to F'$ with respect to bases $B$ and $B'$ for $F$ and $F'$ via the labelled matrix $[\partial]^{B,B'}$.  By slight abuse of terminology, we also call this labeled matrix a presentation of $M$.  

\begin{definition}[Betti Numbers]\label{Def:Min_Pres}
Let $F_\bullet$ be a minimal resolution of a finitely generated $d$-parameter persistence module $M$.  For $i\geq 0$, define the function $\beta_i^M:\Z^d\to \mathbb N$ by 
\[\beta_i^M:=\dimf(F_i/I F_i).\]
For $\z\in \Z^d$, we call $\beta_i^M(\z)$ the \emph{$i^{\mathrm{th}}$ Betti number of $M$ at grade $\z$}.
\end{definition}
It is easily checked that $\beta_i^M(\z)$ is the number of elements at grade $\z$ in any basis for $F_i$.

\begin{remark}\label{Rem:Hilbert_Syzygy}
Hilbert's syzygy theorem tells us that in a minimal resolution $F_\bullet$ of a finitely generated $d$-parameter persistence module $M$, $F_{i}=0$ for $i>d$.  Thus, $\beta_i^M$ is only of interest for $i\leq d$.
\end{remark}

The following formula relating the Hilbert function to the bigraded Betti numbers follows from Hilbert's Syzygy theorem by an easy inductive argument; see \cite[Theorem 16.2]{peeva2011graded} for a proof of the analogous result in the case of $\Z$-graded $K[t_1,\ldots,t_d]$-modules. 

\begin{proposition}\label{Prop:Relation_Between_Betti_And_Dim}
For $M$ a finitely generated $d$-parameter persistence module and $\d\in \Z^d$,
\[\dim M_\d =\sum_{i=0}^d (-1)^i \sum_{\c\leq \d}\beta_i^M(\c).\]
\end{proposition}

\subsection{Gr\" obner bases}\label{Sec:Grobner}
We now define Gr\"obner bases of (homogeneous) submodules of free multi-parameter persistence modules.  Since we work in the multigraded setting, we can give a simpler definition than is otherwise possible, because there is no need to introduce monomial orderings; it suffices to consider ordered bases of free modules.

First, for nonzero $v\in \K^{\nr}$, we define the \emph{pivot} of $v$ to be 
$\max\, \{i \mid v_i\ne 0\}$.  If $v=0$, we define the pivot of $v$ to be $\nll$.  

Now let $F$ be a free finitely generated $d$-parameter persistence module, and fix an ordered basis $B$ of $F$.  For a homogeneous element $m\in F_\d$ such that the pivot of $[m]^B$ is $i$, define the \emph{leading term} of $m$ to be \[\mathrm{LT}(m):=[m]^B_i x^{\z-\sgr{B_i}}B_i\in F_\d.\]
Thus, $[\mathrm{LT}(m)]^B$ is obtained from $[m]^B$ by setting every entry of $[m]^B$ to 0, except for the pivot.

Given a submodule $M\subset F$, define a persistence module $\mathrm{L}(M)\subset F$ by \[\mathrm{L}(M)=\langle \mathrm{LT}(m)\mid m\in M \textup{ homogeneous} \rangle.\]  A homogeneous set $G$ of generators for $M$ is called a \emph{Gr\"obner basis} if \[\mathrm{L}(M)=\langle \mathrm{LT}(g)\mid g\in G\rangle.\]  Note that whether $G$ is a Gr\"obner basis depends on the choice of order on $B$. 

\begin{remark}\label{Rem:Two_Kinds_of_Bases}
For $M\subset F$ a free submodule, a basis for $M$ as defined in \cref{Sec:Free_Modules} needn't be a Gr\"obner basis.  And conversely, a minimal Gr\"obner basis for $M$ needn't be a basis.  Given the potential for confusion here, we will be careful to never forget the modifier ``Gr\"obner'' when referring to Gr\"obner bases.
\end{remark}

\subsection{Graded and Bigraded Matrices}
We define a \emph{graded matrix} to be a matrix with entries in $\K$, with each column labeled by an element of $\Z$, such that the column labels appear in increasing order.  Similarly, we define a \emph{bigraded matrix} to be a matrix with entries in $\K$, with each column labeled by an element of $\Z^2$, such that the column labels appear in colexicographical order.  If $D$ is a (bi)graded matrix, we denote the label of the $j^\mathrm{th}$ column by $\sgr{D}_j$.  
  
  Given a bigraded matrix $D$ and $\z\in \Z^2$, we let $D_{\d}$ (respectively, $D_{\leq \d}$) denote the graded submatrix of $D$ consisting of the columns $j$ of $D$ with $\sgr{D}_j=\z$ (respectively, $\sgr{D}_j\leq\z$);  
here $\leq$ denotes the partial order on $\Z^2$, not the colexicographical order.  For $D$ a graded matrix and $\z\in \Z$, we define $D_\z$ and $D_{\leq \z}$ analogously.

\subsection{Free Implicit Representations: The Input to Our Algorithms}\label{sec:FIRep}
As noted earlier, the algorithms of this paper take as input a bipersistence module $M$ given implicitly as a chain complex of free bipersistence modules
\[\Fa\xrightarrow{\ma} \Fb \xrightarrow{\mb} \Fc\]
with $M\cong \ker \mb/\im \ma$.
We now specify in more detail how we represent this input.  It is clear from the discussion in \cref{Sec:Free_Modules} that with respect to bases $B^X$, $B^Y$, and $B^Z$ for $\Fa$, $\Fb$, and $\Fc$, we can represent the short chain complex above as a pair of matrices $[\ma]$ and $[\mb]$, with the rows and columns of both matrices labelled by elements of $\Z^2$.  In fact, to encode $M$ up to isomorphism, it is enough to keep only the column labels of this pair of matrices: The row labels of $[\mb]$ are not needed, and the row labels of $[\ma]$ are the same as the column labels of $[\mb]$.

In the case that $B^X$ and $B^Y$ are both colexicographically ordered with respect to grade, the column-labeled matrices $[\ma]$ and $[\mb]$ are in fact bigraded matrices.  We then call the pair of bigraded matrices $([\ma],[\mb])$ a \emph{free implicit representation (\firep) of $M$}.
Our algorithms take as input an \firep of $M$.

In the degenerate case that $\Fc$ is trivial, so that $[\mb]$ is an empty matrix, the \firep is simply a presentation for $M$.

\subsection{Column-Sparse Representation of Matrices}
For the complexity analysis of the algorithms of this paper, we assume that matrices are stored in a format allowing for
\begin{itemize}
\item constant time access to the non-zero element of largest index in each column, 
\item $O(\nr)$-time column addition, where $\nr$ is the number of rows in the matrix.
\end{itemize}
Moreover, for practical TDA computations, we need to work with sparse matrix data structures.  To meet these requirements, it suffices to store matrices in a column sparse format, storing the non-zero entries of each column of the matrix as an ordered list.  This is standard in persistence computation \cite{zomorodian2005computing}.  

\begin{remark}
In the context of computing persistent homology, Bauer, Kerber, Reininghaus, and Wagner have studied the practical efficiency of a number of sparse data structures for matrix columns, including linked lists, dynamically allocated arrays, lazy heaps, and (for $\Z/2\Z$ coefficients) bit trees \cite{bauer2017phat}.  They have found that lazy heaps, which perform well when adding a column with very few non-zero entries to a column with many entries, are very effective in practice on TDA problems.  Subsequent implementations of persistent homology computation by these authors use lazy heaps \cite{bauer2014distributed,bauer2017ripser}. 
Following this work, our implementations use lazy heaps as well.  
However, we note that in the worst case, column addition using lazy heaps takes time $O(\nr \log \nr)$, whereas column addition using a list takes time $O(\nr)$.  
\end{remark} 

 \subsection{The Graded Reduction and Kernel Computation in the 1-D Case}\label{Sec:MatRed}
The standard algorithm for computing persistent homology barcodes, introduced by Zomorodian and Carlsson in \cite{zomorodian2005computing}, is a simple matrix reduction algorithm similar to Gaussian elimination.  It is based on column additions.  In this paper, we will call this algorithm the \emph{graded reduction}, or $\StdRed$.  A variant of $\StdRed$ can also be used to compute a basis for the kernel of a morphism of free 1-D persistence modules; this sort of kernel computation is commonly used in TDA to obtain a set of generators for persistent homology.  The graded reduction serves as a starting point for our approach to computing bigraded Betti numbers.  

We now describe the graded reduction and its use in kernel computation.  We will not need to consider how this algorithm is used to compute barcodes, though this is simple; for an explanation, see \cite{zomorodian2005computing} or \cite{edelsbrunner2010computational}.
 
We denote the $j^{\mathrm{th}}$ column of a matrix $R$ by $R(*,j)$, and let $\pivot^R_j$ denote the pivot of this column, as defined in \cref{Sec:Grobner}.  We say $R$ is \emph{reduced} if $\pivot^R_j\ne \pivot^R_{k}$ whenever $j\ne k$ are the indices of non-zero columns in $R$.  Note that if $R$ is reduced, then $\rank R$ is simply the number of non-zero columns of $R$.

$\StdRed$ takes any matrix $D$ and performs left-to-right column additions to transform $D$ into a reduced matrix $R$.  
An outline of the algorithm is given below as \cref{alg:MatRedStd}. 

\begin{algorithm}[h]
  \caption{The Graded Reduction $\StdRed$ (Outline)}\label{alg:MatRedStd}
  \footnotesize
  \begin{algorithmic}[1]
    \Require An $\nr\times \nc$ matrix $D$
    \Ensure A reduced $\nr\times \nc$ matrix $R$ obtained from $D$ by left-to-right column additions
    \State $R\leftarrow D$
    \For{$j=1$ to $\nc$}
        \While{$\exists\ k<j$ such that $\nll\ne\pivot^R_j = \pivot^R_k$}
          \State add $-\frac{R(\pivot^R_j,j)}{R(\pivot^R_j,k)}R(*,k)$ to $R(*,j)$  
        \EndWhile
     \EndFor
  \end{algorithmic}
\end{algorithm}

\begin{remark} In implementing line 1 of Algorithm 1, we do not copy the input matrix $D$ into a new matrix $R$; rather, $R$ is a reference to $D$. We introduce this reference purely as an expository convenience, to distinguish between the input matrix $D$ and the matrices obtained from $D$ by column additions. We use references similarly in the algorithms that follow. 

\end{remark}

To complete the specification of the algorithm $\StdRed$, it remains to explain how we check the conditional of line 3 in \cref{alg:MatRedStd} and how we find $k$ when the conditional does hold. 
This can be done in constant time, provided we maintain a 1-D array $\lows$ of length $\nr$, where $\lows[i]$ records which column reduced so far, if any, has $i$ as its pivot.  We call $\lows$ the \emph{pivot array}.  Our convention is that $\lows$ is indexed starting from 1, not 0.  The full algorithm using the pivot array is given below as \cref{alg:MatRedStdWithPivotArray}.

Later algorithms in this paper use pivot arrays in a similar fashion; we will sometimes suppress the details.

\begin{algorithm}[h]
  \caption{$\StdRed$ (In Detail)}\label{alg:MatRedStdWithPivotArray}
  \footnotesize
  \begin{algorithmic}[1]
    \Require An $\nr\times \nc$ matrix $D$
    \Ensure A reduced $\nr\times \nc$ matrix $R$ obtained from $D$ by left-to-right column additions
        \State $R\leftarrow D$
    \State Initialize an array $\lows$ of size $\nc$, with each entry set to $\nll$
    \For{$j=1$ to $\nc$}
        \While{$R(*,j)\ne 0$ \textbf{and} $\lows[\pivot^R_j]\ne \nll$}
          \State $k \leftarrow \lows[\pivot^R_j]$
          \State add $-\frac{R(\pivot^R_j,j)}{R(\pivot^R_j,k)}R(*,k)$ to $R(*,j)$.  
        \EndWhile
        \If{$R(*,j)\ne 0$}
             \State $\lows[\pivot^R_j]\leftarrow j$
        \EndIf
     \EndFor
  \end{algorithmic}
\end{algorithm}

It is easy to check that for $D$ an $\nr\times \nc$ matrix, $\StdRed(D)$ requires $O(\nr\nc\cdot \min(\nr,\nc))$ elementary operations in the worst case.  

\begin{remark}\label{Rmk:StdRed_and_Rank_Nullity}
Let $\gamma:F\to F'$ be a morphism of free 1-D persistence modules and let $B$, $B'$ be ordered bases for $F$, $F'$  with $B$ in grade-increasing order.  Applying $\StdRed$ to the graded matrix $[\gamma]^{B,B'}$ yields a reduced graded matrix $R$ from which we can read the pointwise ranks and nullities of $\gamma$: For any $\z\in \Z$, $\rank \gamma_\z$ is the number of nonzero columns of $R_{\leq \z}$, and $\nullity \gamma_\z$ is the number of zero columns in $R_{\leq \z}$.  Similarly, $\xi_0^{\ker \gamma}$ can be read off of $R$: $\xi_0^{\ker \gamma}(\z)$ is the number of zero columns in $R_\z$.
\end{remark}

\paragraph{Kernel Computation}
Let $\gamma$, $B$, and $B'$ be as in \cref{Rmk:StdRed_and_Rank_Nullity}.  \cref{alg:SinglyGradedKernel} gives an extension of $\StdRed$ which computes a basis $B_{\ker \gamma}$ for the free module $\ker \gamma$, given $[\gamma]^{B,B'}$ as input.  Each element $b\in B_{\ker \gamma}$ is represented by the pair $(v,\sgr{b})$ where $v$ is the vector in $\K^{|\S|}$ representing $b$, as specified in \cref{Sec:Free_Modules}.  In this algorithm, we maintain an \emph{auxiliary matrix} $V$, initially the $|\Rc|\times |\Rc|$ identity matrix $\Id_{|B|\times |B|}$; every time we do a column operation to $[\gamma]$, we do the same column operation to $V$.  When a column $j$ of $[\gamma]^{B,B'}$ is reduced to $0$, the homogeneous element of $F$ at grade $\sgr{\Rc_j}=\sgr{[\gamma]^{B,B'}}_j$ represented by the $j^{\mathrm{th}}$ column of $V$ is added to $B_{\ker \gamma}$.

It is an easy exercise in linear algebra to check that \cref{alg:SinglyGradedKernel} correctly computes a basis for $\ker \gamma$.

\begin{algorithm}[h]
  \caption{Kernel of a Morphism of Free 1-D Persistence Modules}\label{alg:SinglyGradedKernel}
  \footnotesize
  \begin{algorithmic}[1]
    \Require An $\nr\times \nc$ graded matrix $[\gamma]$ representing a morphism $\gamma$ of free 1-D persistence modules
    \Ensure A basis $B_{\ker \gamma}$ for $\ker \gamma$, represented as a list of pairs $(v,\z)$, where $v\in \K^{\nc}$ and $\z\in \Z$.
    \State $R\leftarrow [\gamma]$
    \State $V\leftarrow \Id_{\nc\times \nc}$
    \State $B_{\ker \gamma}\leftarrow\{\}$
    \For{$j=1$ to $\nc$}
        \While{$\exists\ k<j$ such that $\nll\ne\pivot^R_j = \pivot^R_k$}
          \State add $-\frac{R(\pivot^R_j,j)}{R(\pivot^R_j,k)}R(*,k)$ to $R(*,j)$
          \State add $-\frac{R(\pivot^R_j,j)}{R(\pivot^R_j,k)} V(*,k)$ to $V(*,j)$
        \EndWhile
        \If{$R(*,j)=0$}
             \State append $(V(*,j), \sgr{R}_j)$ to $B_{\ker \gamma}$
        \EndIf
     \EndFor
  \end{algorithmic}
\end{algorithm}

\begin{remark}[Clearing]\label{Rem:Clearing}
Suppose we have graded matrices $[\ma]$ and $[\mb]$ representing the morphisms in a short chain complex of free 1-parameter persistence modules
\[\Fa\xrightarrow{\ma} \Fb \xrightarrow{\mb} \Fc\]
with respect to some choice of bases for $X$, $Y$, and $Z$.  Chen and Kerber \cite{chen2011persistent} have observed that the reduction of $[\ma]$ can be used to expedite the reduction of $[\mb]$.  The key is to note that if $R^f$ and $R^g$ are reduced matrices obtained from $[\ma]$ and $[\mb]$, respectively, by left-to-right column additions, and $R^f(*,j)$ is non-zero with pivot $i$, then $R^g(*,i)=0$.  Hence, for each non-zero column of $R^f$, we can immediately zero out one column of $[\mb]$ before running $\StdRed$ to compute $R^g$.  This shortcut is called the \emph{twist optimization}, or alternatively, \emph{clearing}.  It has been observed that for typical persistent homology computations, this optimization can yield drastic speedups \cite{chen2011persistent,bauer2017phat,bauer2017ripser}.  

Ulrich Bauer has observed that the clearing optimization in fact extends to the computation of a basis for $\ker \mb$.  The idea is simple: If $R^f(*,j)$ is non-zero with pivot $i$ and the $i^{\mathrm{th}}$ column of $[\mb]$ has label $z$, we set $R^g(*,i)=0$ and add the vector in $\Fb_z$ represented by $R^f(*,j)$ to $B_{\ker \mb}$.  Bauer's software Ripser exploits this idea to compute the barcodes of Vietoris--Rips filtrations in a very memory-efficient way \cite{bauer2017ripser}.  
\end{remark}

\section{Kernel Computation in the Bigraded Case}\label{Sec:RankComputationAtAllIndices}
We next present our bigraded reduction algorithm for computing a basis for the kernel of a morphism $\gamma:F\to F'$ of finitely generated free bipersistence modules; as mentioned in the introduction, $\ker \gamma$ is free.  We take the input to be a bigraded matrix $[\gamma]^{B,B'}$ representing $\gamma$ with respect to a choice of ordered bases $B$ and $B'$ for $F$ and $F'$.  Note that the condition that $[\gamma]^{B,B'}$ is bigraded implies that $B$ is colexicographically ordered with respect to grade.

First we will consider the slightly simpler problem of computing $\beta_0^{\ker \gamma}$.  Our algorithm for this extends to an algorithm for computing the kernel itself, in essentially the same way that $\StdRed$ (\cref{alg:MatRedStd,alg:MatRedStdWithPivotArray}) extends to an algorithm for computing a kernel of a morphism of 1-parameter persistence modules (\cref{alg:SinglyGradedKernel}).  We will represent $\beta_0^{\ker \gamma}$ as a list of pairs $(\z,q)\in\Z^2\times \mathbb{N}$ such that $\beta_0^{\ker \gamma}(\z)> 0$, and $q=\beta_0^{\ker \gamma}(\z)$.

\paragraph{Reduction of Bigraded Submatrices}

The bigraded reduction depends on the algorithm $\StdRedSub$ (\cref{alg:BiRedSub}) given below.  $\StdRedSub$ is a variant of  $\StdRed$ which, given a bigraded matrix $D$ and $\z=(z_1,z_2)\in \Z^2$ with $D_{\leq (z_1,z_2-1)}$ already reduced, puts $D_{\leq\z}$ in reduced form.
\begin{algorithm}[h]
  \caption{$\StdRedSub$}\label{alg:BiRedSub}
  \footnotesize
  \begin{algorithmic}[1]
    \Require An $\nr\times \nc$ bigraded matrix $D$; $\z \in \Z^2$ such that $D_{\leq (z_1,z_2-1)}$ is reduced
    \Ensure An $\nr\times \nc$ bigraded matrix $R$ with $R_{\leq\z}=\StdRed(D_{\leq \z})$ and $R(*,j)=D(*,j)$ for $\sgr{D}_j\not\leq \z$.
    \State $R\leftarrow D$
    \State $\mathsf{Indices}\leftarrow\left\{j\in \{1,\ldots,\nc\} \mid \sgr{R}_j=(y,z_2)\textup{ for some }y\leq z_1\right\}$
    \ForAll{$j\in \mathsf{Indices}$, in increasing order,}
        \While{$\exists\ k<j$ such that $\sgr{R}_k\leq \z$ and $\nll\ne\pivot^R_j = \pivot^R_k$}
          \State add $-\frac{R(\pivot^R_j,j)}{R(\pivot^R_j,k)}R(*,k)$ to $R(*,j)$. 
        \EndWhile
     \EndFor
  \end{algorithmic}
\end{algorithm}

As with $\StdRed$, we need to specify how we check the conditional and find $k$ in the \textsf{while} loop for $\StdRedSub$ (line 4).  The way we do this depends on the context in which we call $\StdRedSub$, and will be explained below.

\paragraph{Grids}
Given $Y\subset\Z^2$, define $\grid(Y)\subset \Z^2$ by \[\grid(Y):=\{(z_1,z_2) \mid (z_1,y)\in Y\textup{ and }(x,z_2)\in Y\textup{ for some }x,y\}.\]
For a free module $F$, let $\grid(F):=\grid(\supp(\beta_0^{F}))$, where $\supp$ denotes the support of a function.

\paragraph{Computation of $\beta_0^{\ker \gamma}$ via Bigraded Reduction}
The algorithm $\KerBetti$ (\cref{alg:Bigraded_Reduction}) below computes $\beta_0^{\ker \gamma}$; see \Cref{fig:alg_xy}.  Note that the algorithm makes simultaneous use of both the lexicographical and colexicographical orders on $\Z^2$: It assumes that $\gamma$ is represented as a bigraded matrix, which means that the column labels are in colexicographical order; and it computes $\beta_0^{\ker \gamma}(\z)$ for each $\z\in \grid(F)$, in lexicographical order on $\grid(F)$.  This interplay between the lexicographical and colexicographical orders is crucial to the success of our approach. 

\begin{algorithm}[h]
  \caption{$\KerBetti$: Computes $\beta_0$ of the Kernel of a Morphism of Free Bipersistence Modules}\label{alg:Bigraded_Reduction}
  \footnotesize
  \begin{algorithmic}[1]
    \Require A bigraded matrix $[\gamma]$, representing a morphism $\gamma:F\to F'$ of free bipersistence modules
    \Ensure $\beta_0^{\ker \gamma}$, represented as a list of pairs $(\z,q)\in \Z^2 \times \mathbb{N}$ with $q>0$.
    \State $R\leftarrow [\gamma]$
     \State $\beta_0^{\ker \gamma}\leftarrow \{\}$
    \ForAll{$\z\in \grid(F)$ in lexicographical order}
        \State $R\leftarrow\StdRedSub(R,\z)$  
      \If{$n>0$ columns of $R$ have been reduced to $0$ in the step above}
      \State append $(\z,n)$ to $\beta_0^{\ker \gamma}$.
      \EndIf
    \EndFor
  \end{algorithmic}
\end{algorithm}

\begin{figure}[ht]
  \begin{center}
%
%
\begin{tikzpicture}[scale=0.5,every node/.style={font=\footnotesize}]
  \fill[pattern=thickstripes_pos, pattern color=blue!30!white] (0,0) rectangle (3,3);
  \fill[pattern=thickstripes_neg, pattern color=green!70!white, opacity=0.5] (0,0) rectangle (4,2);
  
  \draw[preaction={clip,postaction={fill=white, draw=orange, line width=4pt}}] (3,2) rectangle +(1,1);
  
  \draw (0,6) -- (0,0) -- (6,0);
  \foreach \i in {1,...,5} {
    \draw (\i, 0) -- (\i, 5.5);
    \draw (0, \i) -- (5.5, \i);
  }
  
  \node[below] at (0.5,0) {$1$};
  \node[below] at (1.5,0) {$2$};
  \node[below] at (2.5,0) {$3$};
  \node[below] at (5.5,0) {$x$};
  \node[left] at (0,0.5) {$1$};
  \node[left] at (0,1.5) {$2$};
  \node[left] at (0,5.5) {$y$};
  \draw[<-] (3.5,-0.1) -- (3.5,-0.6) node[below] {$z_1 = 4$};
  \draw[<-] (-0.1,2.5) -- (-0.6,2.5) node[left] {$z_2 = 3$};
  \node at (6,2.5) {$\cdots$};
  \node at (6,6) {$\iddots$};
  \node at (2.5,6) {$\vdots$};
\end{tikzpicture}
\hspace{36pt}
%
%
\begin{tikzpicture}[yscale=0.65,xscale=0.4,every node/.style={font=\footnotesize}]
  \fill[pattern=thickstripes_pos, pattern color=blue!30!white] (0,0.3) rectangle (3,3.8);
  \fill[pattern=thickstripes_pos, pattern color=blue!30!white] (6,0.3) rectangle (9,3.8);
  \fill[pattern=thickstripes_pos, pattern color=blue!30!white] (12,0.3) rectangle (15,3.8);
  \fill[pattern=thickstripes_neg, pattern color=green!70!white, opacity=0.5] (0,0.3) rectangle (4,3.8);
  \fill[pattern=thickstripes_neg, pattern color=green!70!white, opacity=0.5] (6,0.3) rectangle (10,3.8);
  
  \draw[preaction={clip,postaction={fill=white, draw=orange, line width=4pt}}] (15,0.3) rectangle (16,3.8);
  
  \draw (0,0.3) -- (20,0.3);
  \draw (0,3.8) -- (20,3.8);
  \foreach \i in {1,2,3,4,5,7,8,9,10,11,12,13,14,15,16,17,19}
    \draw (\i,0.3) -- (\i,3.8);
  \foreach \i in {0,6,12,18}
    \draw[thick] (\i,0) -- (\i,4);
  
  \draw[<-,black] (15.5,3.9) -- (15.5,4.7);
  \node at (16.5,5){$\d = (4,3)$};
  \node[below] at (3,0) {$y=1$};
  \node[below] at (9,0) {$y=2$};
  \node[below] at (15,0) {$y=3$};
  \node at (5.5,2) {\scriptsize{$...$}};
  \node at (11.5,2) {\scriptsize{$...$}};
  \node at (17.5,2) {\scriptsize{$...$}};
  \node at (20,2) {$\cdots$};
  
  \draw[decoration={brace,raise=3pt},decorate,red] (12,4) -- (16,4);
  \draw[->, red] (12,5) node[left] {columns to be reduced} to (13.2,5) to [out=0,in=90] (14,4.4);
\end{tikzpicture}
  \end{center}
    \caption{The $\Z^2$-grades through which \cref{alg:Bigraded_Reduction} iterates are shown as squares on the left, and a schematic representation of the bigraded matrix $R$, with $R_{\leq\d}$ highlighted, is shown on the right.  When we begin the reduction of $R_{\leq\d}$ in \cref{alg:Bigraded_Reduction}, line 3, $R_{\leq(z_1,z_2-1)}$ (grades shaded green) is already in reduced form.  However, if $z_2\geq 2$, then $R_{\leq(z_1-1,z_2)}$ (grades shaded purple) is not necessarily reduced, even though it was reduced at an earlier step.  To reduce $R_{\leq\d}$, we only need to reduce columns in $R_{\leq\d}$ with $y$-grade $z_2$ (red bracket).  This is done by calling $\StdRedSub(R,\d)$.} 

\label{fig:alg_xy}
\end{figure}
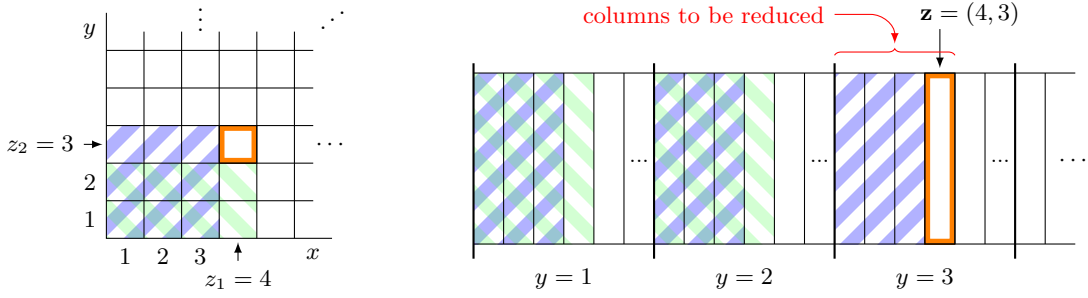

\paragraph{Pivot Arrays for $\KerBetti$}
Our specification of $\StdRedSub$ (\cref{alg:BiRedSub}) above omitted some context-dependent details about the use of a pivot array to implement the \textsf{while} loop of the algorithm. 
We now specify how we handle those details in the context of the algorithm~$\KerBetti$. 

Let $\nr$ be the number of rows of $[\gamma]$.  At the beginning of our call to $\KerBetti$, we initialize a 1-D array $\lows$ of size $\nr$ (indexed from 1, not from 0), with each entry set to $\nll$.  We then implement line 4 of $\StdRedSub$ as follows.  
Let $\ell:=\pivot^R_j$.
\begin{itemize}
  \item If $\ell=\nll$, $\lows[\ell] = \nll$, or $\lows[\ell]\geq j$, then the column index $k$ of line 4 does not exist. If $\ell \ne \nll$ but either $\lows[\ell] > j$ or $\lows[\ell] = \nll$, we set $\lows[\ell] \leftarrow j$. 
  \item Otherwise, $k$ does exist; we take $k=\lows[\ell]$.
\end{itemize}

\begin{proposition}
In the context of $\KerBetti$, the implementation of $\StdRedSub$ described just above works correctly.
\end{proposition}

\begin{proof}
$\KerBetti$ makes one call to $\StdRedSub$ for each $\z\in \grid(F)$.  We say that $\lows$ \emph{is correctly computed at} $\z$ if at the conclusion of the call to $\StdRedSub$ at index $\z$, we have $\lows[\pivot^R_j]=j$ for each non-zero column $R(*,j)$ with $\sgr{R}_j\leq \z$.

It is easy to check that $\KerBetti$'s call to $\StdRedSub$ at index $\z$ works correctly when $\z$ is of the form $\z=(z_1,1)$, and that $\lows$  is correctly computed at such $\z$.  Now fix $z_1$.  By induction on $z_2$, it is similarly easy to check that for each $\z=(z_1,z_2) \in \grid(F)$, $\KerBetti$'s call to $\StdRedSub$ at index $\z$ works correctly, and that $\lows$  is correctly computed at $\z$.
\end{proof}

\begin{remark}[Computation of Pointwise Rank and Nullity]\label{Rem:Pointwise_Rank_And_Nullity_2_Params}
Given $\beta_0^{\ker \gamma}$, one readily obtains $\rankf \gamma$ and $\nullityf \gamma$ at each point of $\grid F$.  The restriction of $\rankf \gamma$ to $\grid F$ completely determines $\rankf \gamma$, and the same is true for $\nullityf \gamma$.  Henceforth, when we speak of computing $\rankf \gamma$ or $\nullityf \gamma$, we will mean doing so at each point of $\grid F$.
\end{remark}

\begin{remark}[Gr\"obner Basis Computation]\label{Rem:Grobner_Basis_Computation}
A simple modification of the algorithm $\KerBetti$ computes a minimal Gr\"obner basis $G$ of $\im \gamma$ with respect to the ordered basis $B'$.  At the start of the algorithm, $G$ is initialized to the empty set.  The procedure for adding elements to $G$ then works as follows: If, while working at index $\d$ in $\grid(F)$, $\KerBetti$ sets $\lows[\ell] \leftarrow j$, then the element of $F'_\d$ represented by $R(*,j)$ is added to $G$.  We leave it to the reader to check that this indeed correctly computes a Gr\"obner basis for $\im \gamma$.
\end{remark}

\paragraph{Kernel Computation}
The algorithm $\KerBetti$ extends to an algorithm $\KerBasis$ (\cref{alg:BigradedKernel} below) for computing a basis for $\ker \gamma$, in essentially the same way that $\StdRed$ extends to an algorithm for computing the kernel in the 1-D setting: We maintain a square auxiliary matrix $V$, initially the identity.  Every time we do a column operation on $R$, we do the same operation on $V$.  If $R(*,j)$ gets reduced to 0 at index $\z$ in the {\sf{for}} loop, we add $(V(*,j),\z)$ to the basis for the kernel.

To give the pseudocode, we need a variant of $\StdRedSub$ (\cref{alg:BiRedSub}) that also performs column operations on an auxiliary matrix.  We call this variant $\StdRedSubAux$ (\cref{alg:StdRedSubAux}).  The calls to $\StdRedSubAux$ by $\KerBasis$ use a pivot array in exactly the same way as the calls to $\StdRedSub$ by $\KerBetti$ do.

\begin{algorithm}[h]
  \caption{$\StdRedSubAux$}\label{alg:StdRedSubAux}
  \footnotesize
  \begin{algorithmic}[1]
    \Require An $\nr\times \nc$ bigraded matrix $D$; an $\nc\times \nc$ matrix $V$; $\z \in \Z^2$ such that $D_{\leq (z_1-1,z_2)}$ is reduced; 
    \Ensure $R=\StdRedSub(D,\z)$; $\bar V=$ the matrix obtained by performing the same column additions on $V$ that we do on $D$
  \end{algorithmic}
\end{algorithm}

\begin{algorithm}[h]
  \caption{$\KerBasis$: Computes Kernel of a Morphism of Free Bipersistence Modules}\label{alg:BigradedKernel}
  \footnotesize
  \begin{algorithmic}[1]
    \Require An $\nr\times \nc$ bigraded matrix $[\gamma]$, representing a morphism $\gamma:F\to F'$ of free bipersistence modules
    \Ensure A basis $B_{\ker \gamma}$ for $\ker \gamma$, represented as a list of pairs $(v,\z)$, where $v\in \K^\nc$ and $\z\in \Z^2$.
    \State $R\leftarrow [\gamma]$
    \State $V\leftarrow \Id_{\nc\times \nc}$
    \State $B_{\ker \gamma}\leftarrow\{\}$
    \ForAll{$\z\in \grid(F)$ in lexicographical order}
        \State $R,V\leftarrow\StdRedSubAux(R,V,\z)$  
        \ForAll{$j$ such that  $R(*,j)$ was first zeroed out in the step above}
        \State append $(V(*,j),\z)$ to $B_{\ker \gamma}$.
        \EndFor
    \EndFor
  \end{algorithmic}
\end{algorithm}

We now verify the correctness of $\KerBasis$.  The proofs of correctness of $\KerBetti$ (\cref{alg:Bigraded_Reduction}) and the variants described in \cref{Rem:Pointwise_Rank_And_Nullity_2_Params} above are very similar.

\begin{proposition}
$\KerBasis$ ($\cref{alg:BigradedKernel}$) correctly computes a basis for $\ker \gamma$.
\end{proposition}

\begin{proof}
To establish the correctness of our algorithm, the key observation is the following: 
\begin{itemize}
\item[$(*)$] For each $\z\in \grid(F)$, when we begin the $\d^{\mathrm{th}}$ iteration of the $\mathsf{for}$ loop in line 4 of \cref{alg:BigradedKernel}, any column that was added to a column of $R_{\leq\d}$ at a previous step of the algorithm was also in $R_{\leq\d}$.
\end{itemize}
For $j$ a column index of $[\gamma]$, let $\g^j=(g^j_1,g^j_2)$ denote $\sgr{[\gamma]}_j$.  To check $(*)$, assume that the $j^\mathrm{th}$ column of $R$ is a column of $R_{\leq\d}$, i.e., that $\g^j\leq \z$.  
Note that for any $\y\in \grid(F)$, if we add $R(*,k)$ to $R(*,j)$ during our call to $\StdRedSubAux(R,V,\y)$, then we must have that $g^k_2\leq g^j_2\leq z_2$, because the column labels are assumed to be colexicographically ordered.  We also have that $g^k_1\leq y_1$.  Moreover, because the algorithm iterates through the indices in lexicographical order, if we call $\StdRedSubAux(R,V,\y)$ before $\StdRedSubAux(R,V,\z)$, then $y_1 \leq z_1$.  Thus $\g^k\leq \z$, which establishes $(*)$.

Given $(*)$, it follows from elementary linear algebra that for each $\z\in \Z^2$, \[B_{\ker f}^\z:=\{\ind^{\z-\sgr{b}}b\mid b\in B_{\ker f},\  \sgr{b}\leq \z\}\] 
is a basis for $\ker f_\z$.  Thus  $B_{\ker f}$ indeed generates $\ker f$.  Moreover, letting $\mathbf{m}$ denote the maximum index in $\grid(F)$, it follows from the linear independence of $B_{\ker f}^{\mathbf{m}}$ that $B_{\ker f}$ is in fact a minimal set of generators for $\ker f$, hence a basis.
\end{proof}

\begin{remark}\label{Rem:Basis_Is_a_Grobner_Basis}
The observations of \cref{Rem:Two_Kinds_of_Bases} notwithstanding, the basis $B_{\ker}$ computed by $\KerBasis$ is clearly a Gr\"obner basis for $\ker \gamma$ with respect to $B$, i.e., distinct elements have distinct pivots.
\end{remark}

\begin{proposition}\label{Prop:Ker_Comp_Complexity}
Given as input an $\nr\times \nc$ matrix representing $\gamma$, $\KerBasis$ runs in 
\[O(\nc(\nr+\nc)\min(\nr,\nc))\] time and requires $O(\nr\nc+\nc^2)$ memory.
\end{proposition}

\begin{proof}
The complexity analysis is similar to that of $\StdRed$, as given in \cite{zomorodian2005computing}. Since each column addition performed by $\KerBasis$ either decreases the pivot of some column of $R$ or reduces the column to zero, $\KerBasis$ performs at most $\nc\cdot \min(\nr,\nc)$ column additions on $R$.  A single column addition on $R$, together with the corresponding operation on the auxiliary matrix $V$, requires $O(\nr+\nc)$ time.
Thus, the column additions performed by $\KerBasis$ require $O(\nc(\nr+\nc)\min(\nr,\nc))$ time in total.

In addition, $\KerBasis$ takes each $j\in \bounds{1}{\nc}$ to be the column index for $\StdRedSub$ at most $\nc$ times.  Thus, the total number of times we change column indices is $\nc^2$.  Each time we change the column index, we do a constant amount of work, beyond that required to perform any column additions.  

Lastly, we perform $O(|\grid(F)|)=O(\nc^2)$ elementary operations simply by iterating through the indices $\z\in \grid(F)$.  The runtime bound now follows.  The bound on memory is clear.
\end{proof}

\begin{remark}\label{Remark:Ker_Betti_Complexity}
The complexity analyses of $\KerBetti$ (\cref{alg:Bigraded_Reduction}) is essentially the same as for $\KerBasis$, except that we do not need to maintain an auxiliary matrix.  Thus, assuming that the input to $\KerBetti$ is an $\nr\times \nc$ matrix, the algorithm runs in time $O(\nr\nc\cdot \min(\nr,\nc)+\nc^2)$, and $O(\nr\nc)$ memory is required.  The variant of $\KerBetti$ for computing a minimal Gr\"obner basis $G$ of $\im(\gamma)$, described in \cref{Rem:Grobner_Basis_Computation}, has the same asymptotic time cost, and requires $O(g+\nr\nc)=O(\nr\nc\cdot \min(\nr,\nc))$ memory, where $g$ is the total number of non-zero terms among all elements of $G$.
\end{remark}

\subsection{Comparison of $\KerBasis$ and Schreyer's Algorithm}\label{Sec:Comparison}
We now compare $\KerBasis$ to the version of Schreyer's algorithm which computes a generating set for $\ker \gamma$, described in  \cite[Chapter 5.3]{cox1998using}.  We denote the latter algorithm as $\KerSchr$.   While  \cite{cox1998using} describes this algorithm in the ungraded setting, \cite{carlsson2010computing} observes that it extends immediately to the multi-graded setting.  We will argue here that $\KerBasis$ is asymptotically more memory efficient than a naive implementation of $\KerSchr$.  

We begin with a brief outline of the algorithm $\KerSchr$, restricting attention to our bigraded setting. The input to $\KerSchr$ is the same as that of  $\KerBasis$, namely an $\nr\times \nc$ bigraded matrix $[\gamma]$ representing $\gamma:F\to F'$ with respect to ordered bases $B$ and $B'$ for $F$ and $F'$, respectively.  
The columns of $[\gamma]$ specify an ordered set of generators $H=\{H_1,\ldots,H_n\}$ for $\im F$.  $\KerSchr$ requires access to the following, which (in a naive implementation, at least) must thus be computed and stored by the algorithm:

\begin{itemize}
\item An ordered Gr\"obner basis $G=\{G_1,\ldots,G_k\}$ for $\im \gamma$ with respect to $B'$.
\item A $\nc \times k$ matrix $T$, whose $j^{\mathrm{th}}$ column expresses $G_j$ as a linear combination of $H$.
\item A $k\times \nc$ matrix $U$, whose $j^{\mathrm{th}}$ column expresses $H_j$ as a linear combination of $G$.
\end{itemize}

Let $F^G$ denote the free bigraded module generated by $G$, and let $\gamma^G:F^G\to \im\gamma$ be the natural map.    Given $G$, \emph{Schreyer's theorem} yields a simple algorithm for computing an ordered set $\{w_1,\ldots,w_l\}$ of homogeneous generators for $\ker \gamma^G$ \cite[Algorithm 1]{erocal2016refined}.   If $G$ is minimal, then it is easily checked that $l<k$.  We may represent $\{w_1,\ldots,w_l\}$ by the columns of a column-labeled $k\times l$ matrix $W$. 
We also regard the matrix $I_{n\times n}-TU$ as a column-labeled matrix, by giving it the same column labels as $[\gamma]$.  Then according to \cite[Chapter 5, Proposition 3.8]{cox1998using}, the columns of the column-labeled matrix
\[\begin{pmatrix}
TW &I_{n\times n}-TU
\end{pmatrix} \]
represent a set of generators for $\ker \gamma$; $\KerSchr$ outputs this matrix.

To rigorously compare $\KerSchr$ and $\KerBasis$, it is helpful to specify some lower-level details of $\KerSchr$ in a way that aligns with those of $\KerBasis$.  We have observed in \cref{Rem:Grobner_Basis_Computation} that a variant of the bigraded reduction efficiently computes a minimal Gr\"obner basis of $G$.  Extending this, the entirety of $\KerSchr$ can in fact be naturally implemented within our bigraded framework, by adapting the ideas of this section.  For example, $T$ can be computed by maintaining an auxiliary matrix as $G$ is computed, as $\KerBasis$ does; and $U$ can be computed simply by recording which column additions are performed.  The details of such an implementation can be specified so that the set of non-zero generators of $\ker \gamma$ computed by $\KerSchr$ is exactly the basis computed by $\KerBasis$.  Yet $\KerSchr$ computes this basis in a needlessly indirect way, and in addition it computes many other trivial generators for $\ker \gamma$ which one knows in advance will be zero.  

Given the parallels between $\KerBasis$ and this implementation of $\KerSchr$, it is easily checked that both the runtime and memory cost $\KerBasis$ are asymptotically bounded above by those of $\KerSchr$.  Intuitively, because $\KerBasis$ is a substantial simplification of $\KerSchr$, one expects $\KerBasis$ to be  substantially more efficient, both in terms of time and memory.  For  memory consumption, we can make this mathematically precise by considering the asymptotic size of the Gr\"obner bases computed by $\KerSchr$.  
As explained above, $\KerSchr$ computes and stores a Gr\"{o}bner basis $G$ for $\im \gamma$.  In contrast, $\KerBasis$ never stores a full Gr\"{o}bner basis for $\im \gamma$.  
In fact, as indicated by \cref{Rem:Grobner_Basis_Computation}, $\KerBasis$ does effectively compute a minimal Gr\"{o}bner basis $G$ for $\im \gamma$ as it computes a basis for $\ker \gamma$, but at most $\nc$ elements of $G$ are stored at a time in the columns of the matrix $R$. 

Thus, together with \cref{Prop:Ker_Comp_Complexity}, the following proposition implies that $\KerBasis$ is asymptotically more memory efficient than the comparable implementation of $\KerSchr$.

\begin{proposition}\label{Prop:Tchernev}
\mbox{}
For $\gamma$ a morphism of free bipersistence modules represented by an $m\times n$ matrix, a minimal Gr\"obner basis for $\im \gamma$ has 
$\Theta(\nc\cdot \min(\nr,\nc))$ elements and a total of $\Theta(\nr\nc\cdot \min(\nr,\nc))$ non-zero terms among all elements, in the worst case.
\end{proposition}

In the case that the field of coefficients $K$ has characteristic 0, \cref{Prop:Tchernev} is due to Alex Tchernev, who showed us the construction used in the proof below.  To obtain the result for fields of positive characteristic, we modify Tchernev's construction.

\begin{proof}[Proof of  \cref{Prop:Tchernev}]
The algorithm of \cref{Rem:Grobner_Basis_Computation} computes a minimal Gr\"obner basis for $\im \gamma$ with $O(\nc\cdot \min(\nr,\nc))$ elements.  With respect to a fixed ordered basis for the codomain of $\gamma$, the number of elements in a minimal Gr\"obner basis $G$ for $\im \gamma$ is easily seen to be independent of the choice of $G$, so any such $G$ has $O(\nc\cdot \min(\nr,\nc))$ elements.  Each element of $G$ has at most $\nr$ terms, so the total number of non-zero terms in $G$ is $O(\nr\nc\cdot \min(\nr,\nc))$.

It remains to show that these bounds are asymptotically tight.  For this, we consider Tchernev's construction, which is the following: For $m,n\geq 1$, let $A$ be the  $m \times n$ matrix given by
\[A_{i,j}=
\begin{cases}
1 &\textup{if } i=m,\\
0 &\textup{if $j=1$ and $i<m$},\\
A_{i,j-1} + A_{i+1,j-1} &\textup{if $j>1$ and $i<m$.}
\end{cases}
\]
 For instance, if $m=5$ and $n=6$, then 
\[
A= \begin{pmatrix}
0 &0 &0 &0 &1 &5 \\
0 &0 &0 &1 &4 &10 \\
0 &0 &1 &3 &6 &10 \\
0 &1 &2 &3 &4 &5 \\
1 &1 &1 &1 &1 &1 
\end{pmatrix}.
\]
As this example illustrates, the non-zero part of $A$ is a truncation of Pascal's triangle.  The matrix $A$ has the following properties:
\begin{itemize}
\item[1.] $A_{i,j}=0$ if and only if $j\leq m-i$.
\item[2.] For all $j\in \{1,\ldots, n-1\}$, $A(*,j+1)-A(*,j)$ is exactly $A(*,j)$ shifted up $1$, with its top entry removed, and a zero in the bottom entry.
\end{itemize}

For any field $K$, let $X$ and $Y$ be free bipersistence modules with coefficients in $K$, where $X$ has a colexicographically ordered basis of size $n$ with grades \[(n,1),(n-1,2),(n-2,3,),\dots,(1,n),\] and
$Y$ has an ordered basis of size $m$ with all grades equal to $(1,1)$.  Fixing such bases, let $\gamma:X\to Y$ be the morphism represented with respect to these bases by the matrix $A$; if $K$ has characteristic $p>0$, then the entries of $A$ are taken mod $p$.

Let $G$ be the minimal Gr\"obner basis of $\im \gamma$ computed via the approach of \cref{Rem:Grobner_Basis_Computation}.  We explicitly describe $G$: The columns of $A$ represent an ordered set \[S_0=(s_1^0,s_2^0,\dots,s_n^0)\] of homogeneous generators for $\im \gamma$.  For $1\leq k\leq \min(m,n)-1$, we inductively define \[S_k=(s^k_1,s^k_2,s^k_3,\ldots,s^k_{n-k})\] by \[s^k_j=x s^{k-1}_{j+1}-y s^{k-1}_j,\]
where $x$ and $y$ are the indeterminates in the polynomial ring $K[x,y]$.
By induction and property 2 above, we see that the column representation of $s^k_j$ in the basis for $Y$ is obtained by shifting the $j^{\mathrm{th}}$ element of $S_0$ up by $k$ rows, with $k$ zeros inserted at the bottom.

 It can be checked that \[G=S_0\cup S_1\cup \cdots \cup S_{\min(m,n)-1}.\]  Thus, the number of elements in $G$ is 
 \begin{align*}
 &\quad n + (n-1) + \cdots + (n-(\min(m,n)-1))  \\
 &= n\cdot \min(m,n)-\min(m,n)(\min(m,n)-1)/2\\
 &=(n-(\min(m,n)-1)/2)\min(m,n)\\
 &=\Theta(\nc\cdot \min(\nr,\nc)).
 \end{align*}

A straightforward calculation shows that if the field $K$ has characteristic 0, then the total number of non-zero terms among all elements of $G$ is  
$\Theta(\nr\nc\cdot \min(\nr,\nc))$.  In brief, this holds because $A$ is a dense matrix.  

In the case that the ground field $K$ has characteristic $p>0$, $G$ is still a minimal Gr\"obner basis for $\im \gamma$, but now $A$ can be asymptotically sparse \cite[Theorem 1 (B)]{singmaster1974notes}, so we no longer have that the total number of terms in $G$ is $\Theta(\nr\nc\cdot \min(\nr,\nc))$.  Thus, to obtain the desired result for such $K$, we modify Tchernev's construction by adding $m$ rows to the top of $A$, in order to make some columns of the Gr\"obner basis dense.  

Specifically, let $v_1$, $v_2$, and $v_3$ denote the column vectors of length $m$ given as follows:
\[
v_1(i)=i \mod 2,\quad v_2(i)=(i+1) \mod 2,\quad v_3(i)=i \mod 3
\]
For example, if $m=5$, then 
\[
v_1= \begin{pmatrix}
1  \\
0  \\
1 \\
0  \\
1  
\end{pmatrix},
\qquad
v_2= \begin{pmatrix}
0  \\
1 \\
0  \\
1  \\
0 
\end{pmatrix},
\qquad
v_3= \begin{pmatrix}
1  \\
0  \\
0 \\
1  \\
0 
\end{pmatrix}.
\]
The key property of these vectors is that any non-zero $K$-linear combination $w$ of these vectors is dense, i.e., at least a constant fraction of entries of $w$ are non-zero.  

Let $B$ be the $m\times n$ matrix whose $j^{\mathrm{th}}$ column is $v_{l}$, where $l\equiv j\mod 3$.  For instance, if $n=6$, then \[B=(v_1\ v_2 \ v_3\ v_1\ v_2\ v_3).\]  Let $X$ be as above, and let $Y'$ be a free module having an ordered basis of size $2m$ with all grades equal to $(1,1)$.  Fix a colexicographically ordered basis for $X$ and a basis for $Y'$, and let $\gamma':X\to Y'$ be the morphism represented with respect to these bases by the block matrix $B\choose A$.  Let $G'$ denote the Gr\"obner basis for $\im \gamma'$ computed via the approach of \cref{Rem:Grobner_Basis_Computation}.

The explicit description of $G$ above carries over to give an explicit description of the subset $H\subset G'$ whose elements have pivot greater than $m$.  In particular, $|H|=|G|=\Theta(\nc\cdot \min(\nr,\nc))$.  (Note that if $m\geq n$, then $H=G'$.)  To finish the proof, we show that total number of non-zero terms among all elements of $G'$ is  
$\Theta(\nr\nc\cdot \min(\nr,\nc))$.  It suffices to show that at least a constant fraction of the columns of $H$ are dense.  

For $j\in \{1,\dots,n\}$, consider the evolution of the $j^{\mathrm{th}}$ column of $B\choose A$ as $H$ is computed via bigraded reduction.  This column contributes $\min(j,m)$ different elements to $H$; we call them \[h_0,\ldots, h_{\min(j,m)-1},\] where the elements are indexed in the order in which they are computed.   Let $b^j_i$ denote the restriction of the column representation of $h_i$ to the first $m$ indices.  Letting $V=(v_1\ v_2\ v_3)$ and  \[
T= \begin{pmatrix}
1 &1 &0   \\
0 &1 &1 \\
1 &0 &1  
\end{pmatrix}, 
\]
one can check that $b^j_i$ is column $l$ of the matrix product $VT^i$, where $l\equiv j\mod 3$.  Thus, $b^j_i$ is dense if and only if column $l$ of $T^i$ is nonzero.  To finish the proof, it therefore suffices to observe that $T$ is not nilpotent, because then at least a constant fraction of the elements of $H$ are dense.  To check that $T$ is not nilpotent, we use the standard fact that the characteristic polynomial $P(\lambda)$ of a nilpotent $n\times n$ matrix is given by  $P(\lambda)=\lambda^n$.  The Cayley--Hamilton theorem then implies that $T$ is nilpotent if and only if $T^3=0$.  But it is easily checked that $T^3\ne 0$ over any field $K$, so $T$ is not nilpotent.
\end{proof}

\section{Computing a (Minimal) Presentation}\label{Sec:Computing_A_Presentation}
As noted in the introduction, our algorithm for computing a minimal presentation from an \firep first computes a (not necessarily minimal) presentation, and then minimizes it.  
In fact, our algorithm first computes a presentation $P$ such that the non-minimal summands are of the form $G\xrightarrow{\Id_G} G$, for $G$ free; we will call such a presentation \emph{semi-minimal.}
In this section, we present the details, and also give an algorithm for computing the Betti numbers directly from a semi-minimal presentation, without minimizing.

\subsection{Computing a Semi-Minimal Presentation}\label{Sec:Semi-Minimal_Pres}
Suppose 
\[\Fa\xrightarrow{\ma} \Fb \xrightarrow{\mb} \Fc\]
is a chain complex with $M\cong \ker \mb/\im \ma$, and we are given an \firep 
$([\ma],[\mb])$, with respect to some choice of bases for $X$, $Y$, and $Z$.
Our algorithm for computing a semi-minimal presentation $P$ for $M$ proceeds in three steps.

\begin{itemize}
\item[1.] From $[\ma]$, find a minimal ordered set of generators $S$ for $\im \ma$ using the algorithm $\MinGens$ (\cref{alg:Minimal_Gens}) below.
\item[2.] From $[\mb]$, compute an ordered basis $B_{\ker}$ for $\ker \mb$, using $\KerBasis$ (\cref{alg:BigradedKernel}).  
\item[3.] Express each element of $S$ in $B_{\ker}$-coordinates, as in \cref{Eq:Linear_Combination}; put the resulting column vectors into a matrix $P$, with column labels the grades of $S$ and row labels the grades of $B_{\ker}$.
\end{itemize}

Since the columns of $[\ma]$ already represent a generating set for $\im \ma$, one can compute $S$ via the bigraded reduction, using a slight variant of \cref{alg:Bigraded_Reduction}.  The algorithm $\MinGens$ does exactly this.
\begin{algorithm}[h]
  \caption{$\MinGens$: Computes a minimal set of generators for the image of a morphism of free bipersistence modules}\label{alg:Minimal_Gens}
  \footnotesize
  \begin{algorithmic}[1]
    \Require An $\nr\times \nc$ bigraded matrix $[\gamma]$, representing a morphism $\gamma:F\to F'$ of free bipersistence modules
    \Ensure A minimal set $S$ of generators of $\im \gamma$, represented as a list of pairs $(v,\z)$, where $v\in \K^\nr$ and $\z\in \Z^2$.
    \State $R\leftarrow [\gamma]$
    \State $S\leftarrow\{\}$
    \ForAll{$\z\in \grid(F)$, in lexicographical order}
        \State $R\leftarrow\StdRedSub(R,\z)$
      \ForAll{columns $R(*,j)$ of $R_\z$ not reduced to $0$ in the last step} 
      \State append $(R(*,j),\z)$ to $S$.
      \EndFor
    \EndFor
  \end{algorithmic}
\end{algorithm}

Note that even if step 1 is omitted, steps 2 and 3 still yield a presentation for $M$, but this presentation may not be semi-minimal.  

Step 3 is just ordinary linear algebra.  Since as noted in \cref{Rem:Basis_Is_a_Grobner_Basis}, $B_{\ker}$ is in fact a Gr\"obner basis  for $\im(\gamma)$ with respect to the given basis for $\Fb$, the computation is especially simple.  It can be carried out efficiently in the column-sparse setting using a pivot array.   We leave the easy details to the reader.

By essentially the same complexity analysis used to prove \cref{Prop:Ker_Comp_Complexity}, we have that if $[\mb]$ and $[\ma]$ have dimensions $a\times b$ and $b\times c$, respectively, then computing a semi-minimal presentation of $M$ using this algorithm requires 
\[O(b(a+b)\min(a,b)+bc\min(b,c)+c^2)=O((a+b+c)^3)\] time and $O(b(a+b+c))$ memory.  The semi-minimal presentation has at most $b$ rows and at most $c$ columns.

\begin{remark}
Our algorithm in fact computes a semi-minimal presentation where the orders of both the row and column labels are compatible with the partial order on $\Z^2$.  In  \cref{Sec:Minimizing_Pres,Sec:Betti_from_Pres} below, it will be useful to assume that the labels are ordered in this way.
\end{remark}

\begin{remark}[Clearing]\label{Rem:Clearing_2D}
We can the leverage work done in Step 1 of the presentation computation to expedite Step 2, using a variant of the clearing optimization described in \cref{Rem:Clearing}.  In its simplest form, this 2-parameter clearing yields an element of $\ker \mb$ for every column added to $S$ whose label is equal to the label of its pivot.  In the 2-parameter setting, a more aggressive variant of clearing is also possible.  

As of this writing, clearing is not yet implemented in our code, and it remains to be seen whether clearing can lead to the same sort of drastic speedups in the 2-parameter setting that it does in the 1-parameter setting.
\end{remark}

\begin{remark} \label{Rmk:Hilbert_Function}
Since $\dimf(M)= \nullityf \mb-\rankf \ma$, steps 1 and 2 of the above algorithm for computing a semi-minimal presentation from an \firep can be modified slightly to also compute the Hilbert function $\dimf(M)$ with negligible additional work; see \cref{Rem:Pointwise_Rank_And_Nullity_2_Params}.
\end{remark}

\subsection{Minimizing a Semi-Minimal Presentation}\label{Sec:Minimizing_Pres}
It is well known in commutative algebra that a non-minimal resolution can be minimized using a variant of Gaussian elimination.  This is explained, e.g., in \cite[pages 127 and 166]{greuel2012singular}.   To minimize a semi-minimal presentation, the simple procedure $\MinimizePres$ (\cref{alg:Minimize_Pres}) below is sufficient.

\begin{algorithm}[h]
  \caption{$\MinimizePres$: Minimizes a semi-minimal presentation}\label{alg:Minimize_Pres}
  \footnotesize
  \begin{algorithmic}[1]
    \Require An $\nr\times \nc$ semi-minimal presentation $P$ for $M$, with the orders of the rows and columns compatible with the partial order on $\Z^2$
    \Ensure A minimal presentation $R$ for $M$
    \State $R\leftarrow P$
    \For{$j=1$ to $\nc$}
    	 \State $p\leftarrow \pivot^R_j$
        \If{the label of column $j$ is equal to the label of row $p$}
         \For{$k=j+1$ to $\nc$}
         \State add $-\frac{R(p,k)}{R(p,j)}R(*,j)$ to $R(*,k)$
         \EndFor
         \State remove column $j$ and row $p$ from $R$.
         \EndIf
      \EndFor
   \State \Return the labeled matrix $R$
  \end{algorithmic}
\end{algorithm}

The \textsf{for} loop of lines 5--6 has the effect of zeroing out all entries of row $p$ except $R(p,j)$.  (Note that at the start of the \textsf{for} loop, we already have $R(p,k)=0$ for all $k<j$.)
When working with column-sparse matrices, the removal of the row $p$ in line 7 of $\MinimizePres$ is implicit; we simply set $R(p,j)=0$.  When the reduction is complete, we reindex all of the entries of the matrix to account for the rows that have been removed.  

We leave to the reader the straightforward proof that \cref{alg:Minimize_Pres} correctly computes a minimal presentation.

$\MinimizePres$ differs from the other algorithms we have considered thus far, in that it requires us to access non-pivot entries of a column (namely, the entries $R(p,k)$, in line 6).  When implemented using a dense matrix, this access is constant time.   
 However, in practice, we want to avoid using dense matrices, and instead work with column-sparse matrices.    
RIVET currently addresses this by representing each column in the input to $\MinimizePres$ as a dynamically allocated array; a binary search then allows us to to access a non-pivot element in time logarithmic in the number of entries of the column.  With this implementation, the cost of $\MinimizePres$ is \[O(\nc^2\log \nr+ \nr\nc\cdot \min(\nr,\nc))=O(\nc^2\nr).\]  

$\MinimizePres$ is embarrassingly parallel.  Specifically, the \textsf{for} loop of lines 5--6, which dominates the cost of the algorithm, can be parallelized.  Thanks to work of Bryn Keller and Dave Turner, our implementation of $\MinimizePres$ in RIVET implements this parallel computation, and we have found that this leads to significant speedups.  For example, on one fairly large example of interest, the parallel implementation led to a factor of 12 speedup of $\MinimizePres$ on a machine with 16 cores.

\begin{remark}\label{Rem:Kerber}
Michael Kerber has pointed out to us that an alternate approach, described in \cite{fugacci2018chunk}, allows us to minimize the presentation in the column-sparse setting without using binary search.  In brief, similar to the standard reduction algorithm for computing persistent homology, one can carry out the minimization in such a way that all of the operations performed on column $j$ are done at once.   This minimization strategy is used to minimize presentations of bipersistence modules in \cite{kerber2021fast}.
\end{remark}

\subsection{Computing the Betti Numbers from a Semi-Minimal Presentation}\label{Sec:Betti_from_Pres}
The row and column labels of a minimal presentation of $M$ encode $\beta_0^M$ and  $\beta_1^M$, respectively.  Given these and $\dimf(M)$, \cref{Prop:Relation_Between_Betti_And_Dim} yields $\beta_2^M$.  
However, in cases where a minimal presentation is not needed, it is more efficient to compute $\beta_0^M$ and  $\beta_1^M$ without fully minimizing the presentation.  This is in keeping with the general principle that computing Betti numbers from a resolution is easier than minimizing the resolution; see, e.g., the timing results in \cite[Section 6]{erocal2016refined}.  Here, we give a simple algorithm which takes as input a semi-minimal presentation $P$ for $M$ and outputs $\beta_0^M$ and  $\beta_1^M$, without minimizing the presentation.  This is a variant of a standard approach in commutative algebra for computing Betti numbers from a non-minimal resolution.
(The algorithm is not yet implemented in RIVET.) 

 Let $\partial:F\to F'$ be a semi-minimal presentation of a module $M$, and let $[\partial]$ be the matrix representation of $\partial$ with respect to bases $B$ and $B'$ for $F$ and $F'$.  For $\z\in \Z^2$, let $D^\z$ denote the submatrix of $[\partial]$ formed by the columns and rows with label $\z$, and let $m_\z$ and $n_\z$ be the number of rows and columns of $D^\z$, respectively.  
 
Our approach hinges on the following result:
\begin{proposition}\label{Prop:Betti_From_Min_Pres}
For all $\z\in \Z^2$, we have
\begin{align*}
\beta^M_0(\z)&=m_\z-\rank D^\z,\\
\beta^M_1(\z)&=n_\z-\rank D^\z=
\nullity D^\z.
\end{align*}
\end{proposition}

\begin{proof}
Since $\partial$ is semi-minimal, $\partial\cong \delta\oplus \Id_G$, where $\delta$ is a minimal presentation for $M$ and $\Id_G:G\to G$ is the identity map of a finitely generated free bipersistence module $G$.  It suffices to show that $\rank D^\z=\xi_0^G(\z)$. 
For $N$ a bipersistence module, let $\bar N$ denote the vector space $(N/IN)_\z$, where $IN$ is as defined in \cref{Sec:BettiNumDef}.  Note that $\partial$ induces a map $\bar \partial:\bar F\to \bar F'$.  It is immediate from the definition of a minimal presentation that $\rank \bar \partial=\xi_0^G(\z)$.  Further, note that $B$ and $B'$  induce bases for $\bar F$ and $\bar F'$: Simply take the image under the quotient map of the basis elements with grade $\z$.  Moreover, $\bar \partial$ is represented with respect to these bases by the matrix $D^\z$.  Therefore $\rank D^\z=\xi_0^G(\z)$, as desired. 
\end{proof}

In view of \cref{Prop:Betti_From_Min_Pres},  to compute $\beta^M_0$ and $\beta^M_1$ from $[\partial]$, essentially all we need to do is compute the rank of each $D^\z$.  If the order of row labels of $[\partial]$ is compatible with the partial order on $\Z^2$, and columns with the same label are ordered consecutively, then this is particularly easy to do in the column-sparse framework.   Moreover, different values of $\z$ can be handled in parallel.  We leave the straightforward details to the reader.  If $[\partial]$ is an $\nr\times \nc$ matrix, then computing $\xi_0$ and $\xi_1$ from $[\partial]$ using this algorithm requires $O(\nr\nc\cdot \min(\nr,\nc))$ time and $O(\nr\nc)$ memory, assuming the rows and columns of $[\partial]$ have been ordered appropriately.

\subsection{Computational Complexity of Computing Betti Numbers from an \firep}\label{Sec:Betti_from_FI_Rep}
Using the complexity bounds given above, we now bound the complexity of our approach to computing the bigraded Betti numbers of a finitely generated bipersistence module $M$, given an \firep 
$([\ma],[\mb])$ for $M$: If $[\mb]$ and $[\ma]$ have dimensions $a\times b$ and $b\times c$, respectively, then computing $\beta^M_0$ and $\beta^M_1$ using the approach of \cref{Sec:Betti_from_Pres} takes 
\[O(b(a+b)\min(a,b)+bc\min(b,c)+c^2)\] 
time and $O(b(a+b+c))$ memory; indeed, the cost is dominated by the cost of computing a semi-minimal presentation for $M$, using the algorithm of \cref{Sec:Semi-Minimal_Pres}.  If we wish to also compute $\dimf(M)$ along the way, as described in \cref{Rmk:Hilbert_Function}, then this requires an additional $O((b+c)^2)$ memory to store $\dimf(M)$.  Computing $\beta^M_2$ from $\beta^M_0$, $\beta^M_1$, and $\dimf(M)$ requires an additional $O(c^2)$ elementary operations and $O(c)$ memory.  Thus, computing the bigraded Betti numbers from an \firep via our approach requires 
\[O(b(a+b)\min(a,b)+bc\min(b,c)+c^2)=O((a+b+c)^3)\] 
time and 
\[O(b(a+b+c)+c^2)=O((a+b+c)^2)\]
memory.

\section{Experiments}\label{Sec:Experiments}
We report on computational experiments with implementations of algorithms for computing Betti numbers in RIVET, Macaulay2, and Singular.  

\paragraph{Description of the Test Data}
We perform tests on {\firep}s (i.e., short chain complexes) arising from 12 different data sets.  Of these, six are synthetic point clouds, of sizes ranging between 50 and 800 points.  Each point cloud is a noisy sample of an annulus in $\R^2$. 

Specifically, to construct a point cloud $X$ of $n$ points, we take an i.i.d.\ sample of $0.9n$ points, given in polar coordinates as follows: The radius is sampled from a Gaussian distribution with mean $2$ and standard deviation $0.3$, and the angle is chosen uniformly on the circle.
We then sample an additional $0.1n$ points from a uniform distribution on the rectangle $[-6,6]\times[-6,6]$.  We endow $X$ with a density function $f:X\to \R$ by taking $f(x)$ to be the number of points in $X$ within distance $1$ of $x$.
From this data, we construct a \emph{density-Rips bifiltration}, as described in \cite{carlsson2009theory}, using the Euclidean distance between points.  

The other six data sets we consider are finite metric spaces (i.e., distance matrices) from a publicly available collection of data sets assembled by the authors of the paper ``A Roadmap for the Computation of Persistent Homology'' for benchmarking the performance of 1-parameter persistent homology software  \cite{otter2017roadmap,otter2017github}.  This collection, which we will call the \emph{Roadmap benchmark}, contains 23 finite metric spaces; for our experiments we considered just the first six, in alphabetical order of the file names.  Following the notation of \cite{otter2017roadmap}, where these data sets are described in detail, we denote the data sets as \textsf{eleg}, \textsf{drag\ 1}, \textsf{drag\ 2},  \textsf{frac\ l}, \textsf{frac\ r}, and \textsf{frac\ w}.    

For each finite metric space, we construct a density-Rips bifiltration as we did for the synthetic point cloud data.  As indicated above, the construction of the bifiltration depends on a choice of density function, which in turn depends on a distance parameter $r$, taken to be 1 above.  We use a different value of $r$ for each data set in the Roadmap benchmark.  Specifically, we take $r$ to be the $20^{\mathrm{th}}$ percentile of all non-zero distances between pairs of points.  

For each of the 12 bifiltrations thus constructed, we use RIVET to compute the {\firep}s of the degree 0 and 1 homology modules, taking the field $\K$ to be $\Z/2\Z$.\footnote{Experiments not reported here indicate that the choice of (reasonably small) prime field has only a minor impact on the speed of computations in RIVET, provided we follow standard practice and precompute a table of logarithms in order to expedite the field arithmetic.}

\paragraph{Summary of Tests}
For each \firep, we ran Betti number computations in RIVET (version 1.0), Macaualy2 (version 1.12), and Singular (version 4.1.1).  A Linux machine with 8 cores and 64 GB of RAM was used.  All computations were performed using a single core, except for RIVET's minimization of a semi-minimal presentation, which runs in parallel.  

We tested two routines to compute Betti numbers in Macaulay2.  The first computes a minimal resolution via the algorithm of LaScala and Stillman, which immediately yields the Betti numbers.  This routine is run on an \firep via the following sequence of Macaulay2 commands
\begin{itemize}
\item[] \textsf{homology $\to$ minimalPresentation $\to$ resolution $\to$ betti}.
\end{itemize}

The second routine uses an algorithm recently added to Macaulay2 for fast computation of non-minimal resolutions and Betti numbers.  This is run via the commands
\begin{itemize}
\item[] \textsf{homology $\to$ minimalBetti.}
\end{itemize}

In the first step of both of the above routines, we call the Macaulay2 function \textsf{coker} instead of the function  \textsf{homology} when the input represents a degree-0 homology module.  (In this case the \firep consists of a single non-zero matrix, so a call to \textsf{coker} is sufficient.)

We also tested two different routines to compute Betti numbers in Singular.  The first uses the Singular function \textsf{res} to compute a resolution and then computes Betti numbers from this.  
This is called on an \firep via the Singular commands
\begin{itemize}
\item[] \textsf{homology $\to$ res $\to$ betti.}
\end{itemize}
The second routine uses the recent refinement of Schreyer's algorithm \cite{erocal2016refined} to compute a free resolution.  This is run via the Singular commands
\begin{itemize}
\item[] \textsf{homology $\to$ groebner $\to$ fres $\to$ betti.}
\end{itemize}
In both of the above routines, it is not necessary to call \textsf{homology} when the input represents a degree-0 homology module.

Additional details about these computations are given below.

\paragraph{Methodological Details}
Differences between the software packages led to some issues in comparing the performance of RIVET, Macaualy2, and Singular.  Here we describe these issues and how we addressed them.  
First, we need a definition:
  
\begin{definition}[$\Z$-Grading]\label{Def:Z_Grading}
A \emph{$\Z$-graded} $k[x,y]$-module is a $k[x,y]$-module $N$ with a $k$-vector space decomposition $N\cong \oplus_{z\in \Z} N_z$ such that for each monomial $m$ in $k[x,y]$ of total degree $d$, $m(N_z)\subseteq N_{z+d}$ for all $z\in \Z$.  
\end{definition}

Any bigraded $k[x,y]$-module $M$ determines a $\Z$-graded $k[x,y]$-module $\bar M$, by taking $\bar M_z=\oplus_{a+b=z} M_{a,b}$.  The \emph{$\Z$-graded Betti numbers} of $\bar M$ are the sums along diagonals of the bigraded Betti numbers of $M$.  Thus, any algorithm for computing bigraded Betti numbers also yields $\Z$-graded Betti numbers.

While both Macaulay2 and Singular offer some functionality for computing bigraded Betti numbers, the functions \textsf{minimalBetti} in Macaulay2 
 and \textsf{fres} in Singular currently only work in the $\Z$-graded setting.  Moreover, Singular's \textsf{homology} function does not handle bigraded input.\footnote{In fact, even in the $\Z$-graded case there is a bug in Singular's \textsf{homology} function that causes all of the grades of generators in the computed presentation to be set to 0.  Because we are only interested in performance and not in the output, we choose to nevertheless report timing results for these incorrect computations, as we expect that the computational cost of the correct computations would be very similar.}  
 
 Thus, we do not report on any bigraded homology computations in Singular, but only on $\Z$-graded computations, and our computations using \textsf{minimalBetti} in Macaulay2 are also done only in the $\Z$-graded setting.  
 
 For the algorithm of La Scala and Stillman, we report runtimes for both $\Z$-graded and bigraded versions of the computations.  The timing results we obtain are very similar for both versions.  We expect that for each of  the other 
 $\Z$-graded computations we do, the computational cost of the analogous bigraded computation would be similar.

\begin{table}[ht!]{\footnotesize
  \caption{Runtimes in seconds for our Betti experiments in RIVET, Macaulay2 ($\Z$-graded, bigraded, and \textsf{minimalBetti} tests), and Singular (using \textsf{res} and \textsf{fres}).  For each data set, the results for both $0^{\mathrm{th}}$ and $1^{\mathrm{st}}$ homology are given, in the upper and lower rows, respectively.  A dash indicates that the computation caused a crash and returned no result.  An asterisk indicates that the computation ran out of main memory, began using swap memory, and was stopped before completing.}
  \label{runtimes}

  \begin{center}\begin{tabular}{c|c|c|ccc|cc}
    \toprule
     &  &  & \multicolumn{3}{c|}{Macaulay2} & \multicolumn{2}{c}{Singular} \\
    Data  & $\#$ of Points & RIVET & $\Z$-graded & Bigraded & \textsf{minimalBetti} & \textsf{res} & \textsf{fres} \\
    \midrule
        \multirow{12}{*}{\textsf{point cloud annuli}}    
            &\multirow{2}{*}{50} & 0.001 & 0.059 & 0.073 & 0.020 & 0.010 & 0.010 \\
            &                    & 0.049 & 7.278 & 7.175 & 0.993 & 11.45 & 11.48 \\
            \cmidrule(l){2-8}
            &\multirow{2}{*}{100} & 0.009 & 0.476 & 0.448 & 0.176 & 0.135 & 0.155 \\
            &                     & 0.459 &  ---  &  ---  &  ---  & 1219  & 1223  \\
            \cmidrule(l){2-8}
            &\multirow{2}{*}{200} & 0.058 & 4.457 & 4.464 & 1.348 & 1.765 & 1.890 \\			
            &                     & 8.346 &  ---  &  ---  &  ---  &  ---  &  ---  \\
            \cmidrule(l){2-8}
            &\multirow{2}{*}{400} & 0.300 & 46.56 & 46.22 & 11.64 &  ---  &  ---  \\
            &                     & 137.3 &  ---  &  ---  &  ---  &  ---  &  ---  \\
            \cmidrule(l){2-8}
            &\multirow{2}{*}{600} & 0.822 &  ---  &  ---  &  ---  &  ---  &  ---  \\
            &                     & 640.5 &  ---  &  ---  &  ---  &  ---  &  ---  \\
            \cmidrule(l){2-8}
            &\multirow{2}{*}{800} & 1.549 &  ---  &  ---  &  ---  &  ---  &  ---  \\
            &                     &   *   &  ---  &  ---  &  ---  &  ---  &  ---  \\	
    \midrule
        \multirow{2}{*}{\textsf{eleg}} & \multirow{2}{*}{297} & 0.070 & 7.063 & 6.846 & 1.615 & 6.000 & 6.795 \\
                                       &                      & 4.283 &  ---  &  ---  &  ---  &  ---  &  ---  \\
        \cmidrule(l){1-8}
        \multirow{2}{*}{\textsf{drag\ 1}} & \multirow{2}{*}{1000} & 1.103 &  ---  &  ---  &  ---  & 1473  & 1572  \\
                                          &                       &   *   &  ---  &  ---  &  ---  &  ---  &  ---  \\
        \cmidrule(l){1-8}
        \multirow{2}{*}{\textsf{drag\ 2}} & \multirow{2}{*}{2000} & 5.570 &  ---  &  ---  &  ---  & 24905 & 26234 \\
                                          &                       &   *   &  ---  &  ---  &  ---  &  ---  &  ---  \\
        \cmidrule(l){1-8}
        \multirow{2}{*}{\textsf{frac\ l}} & \multirow{2}{*}{512} & 1.169 & 176.4 & 174.0 & 31.18 & 20.50 &  ---  \\
                                          &                      & 273.6 &  ---  &  ---  &  ---  &  ---  &  ---  \\
        \cmidrule(l){1-8}
        \multirow{2}{*}{\textsf{frac\ r}} & \multirow{2}{*}{512} & 0.382 & 62.71 & 60.50 & 8.133 & 54.46 & 60.26 \\
                                          &                      & 168.0 &  ---  &  ---  &  ---  &  ---  &  ---  \\
        \cmidrule(l){1-8}
        \multirow{2}{*}{\textsf{frac\ w}} & \multirow{2}{*}{512} & 0.020 & 3.707 & 3.923 & 0.406 & 43.77 & 49.42 \\
                                          &                      & 76.59 &  ---  &  ---  &  ---  &  ---  &  ---  \\      
    \bottomrule
  \end{tabular}\end{center}
  }
\end{table}

\paragraph{Runtime Results}
\Cref{runtimes,Table:MinPresTimings} display the results of our experiments.  Note that while the number $n$ of points in each data set is given in the table, the size of the \firep is much larger than $n$.
Namely, the non-zero matrix in the \firep for $H_0$ has dimensions $n \times \binom{n}{2}$, with two non-zero entries per column, and the matrices in the \firep for $H_1$ have dimensions $n \times \binom{n}{2}$ and $\binom{n}{2} \times \binom{n}{3}$, with two and three non-zero entries per column, respectively.
The timing results do not include the time to compute the \firep from the point cloud data (which was done in RIVET, in all cases), but this step does not contribute significantly to the total cost of the computations.

\begin{table}[ht]{\footnotesize
  \caption{Time in seconds to compute a semi-minimal presentation and to minimize this.  Minimization was performed in parallel using an 8-core machine.}
  \label{Table:MinPresTimings}
  \begin{center}\begin{tabular}{c|ccc}
    \toprule
    Data  & $\#$ of Points & Semi-Minimal Presentation & Minimization \\
    \midrule
    \multirow{12}{*}{\textsf{point cloud annuli}}    
                                     &\multirow{2}{*}{50}       & 0.001 &0.000 \\
    					&	          	        &  0.028 &0.009 \\
							         \cmidrule(l){2-4}
				   &\multirow{2}{*}{100}	& 0.008 & 0.001 \\
				   	&				& 0.431 & 0.028 \\
									 \cmidrule(l){2-4}
				    &\multirow{2}{*}{200}    & 0.045 &0.005 \\			
               				&				& 7.969& 0.238  \\
								 \cmidrule(l){2-4}
               		           &\multirow{2}{*}{400}      & 0.231&0.029 \\
               				&				& 138.3 & 5.858  \\
								 \cmidrule(l){2-4}
				  &\multirow{2}{*}{600}       & 0.686 &0.048  \\
				     &                                   &  502.8&38.89 \\
				                                           \cmidrule(l){2-4}
				   &\multirow{2}{*}{800} 	& 1.329&0.083 \\
					&				& * & * \\	
					        \midrule
        \multirow{2}{*}{\textsf{eleg}}             & \multirow{2}{*}{297}                  &0.056	&0.005\\
                                                                  &                                                 &3.293	 &0.979  \\
            \cmidrule(l){1-4}
        \multirow{2}{*}{\textsf{drag\ 1}}           & \multirow{2}{*}{1000}             & 0.973	&0.042  \\
                                                                  &         				       & * & * \\
           \cmidrule(l){1-4}
       \multirow{2}{*}{\textsf{drag\ 2}}          & \multirow{2}{*}{2000}              & 5.271	&0.103 \\
                                                                   &         				      & * & * \\
            \cmidrule(l){1-4}
        \multirow{2}{*}{\textsf{frac\ l}}              & \multirow{2}{*}{512}               & 1.169	&0.121 \\
                                                                    &            				        & 254.1	&19.33 \\
         \cmidrule(l){1-4}
         \multirow{2}{*}{\textsf{frac\ r}}           & \multirow{2}{*}{512}                 & 0.307	 &0.026 \\
                    						&                                                & 151.4 & 16.53 \\
             \cmidrule(l){1-4}
          \multirow{2}{*}{\textsf{frac\ w}}          & \multirow{2}{*}{512}                &0.020	&0.000 \\
              							&	  					&68.09	&8.498  \\      
    \bottomrule
  \end{tabular}\end{center}
  }
\end{table}

In all of our experiments, the RIVET computations were much faster than the corresponding computations in Macaulay2 and Singular, and neither Macaulay2 nor Singular was able to handle our larger problem instances.   
To the best of our understanding, Macaulay2 does not currently handle sparse matrices in the context of Betti number computation; in view of this, it is not surprising that Macaulay2 does not handle our larger problems.   Singular does work with sparse matrices, but nevertheless does not handle our larger problems.  

It turns out that for all our Singular computations in homology degree 1, the cost of calling the \textsf{homology} function dominates the total cost of the computation.  This explains why, for homology degree 1, there is not much difference between the timings for the $\textsf{res}$ and $\textsf{fres}$ computations in Singular.  
\cref{Table:MinPresTimings} shows that the cost of computing the minimal presentation (using parallelization) was usually much less than the cost of computing the semi-minimal presentation.

\section{Discussion}\label{Sec:Discussion}
To the best of our knowledge, our work is the first to give algorithms for the computation of minimal presentations and bigraded Betti numbers which have cubic runtime and quadric memory requirements, with respect to the sum of the dimensions of the input matrices.  
Our experiments indicate that our algorithms perform well enough in practice to be used on many of the same kinds of data sets to which one usually applies 1-parameter persistent homology.  Still, there is still much room for improvement in the implementation of our algorithms.   As mentioned in the introduction, some such improvements have been introduced in the recent work of Kerber and Rolle \cite{kerber2021fast}.    
We expect that for certain types of bifiltrations, such as the Vietoris--Rips bifiltrations considered in \cite{carlsson2009theory} and \cite{lesnick2015interactive}, additional optimizations along the lines of those used in Ripser \cite{bauer2017ripser} may also be useful for computing presentations of larger data sets.    

As indicated in the introduction, to work with a homology module of a multi-critical bifiltration, our current approach is to first construct a chain complex of free modules, following Chacholski et al. \cite{chacholski2017combinatorial}.  However, for some multi-critical bifiltrations of particular interest to us, such as the \emph{degree-Rips} bifiltrations introduced in \cite{lesnick2015interactive} and studied in \cite{blumberg2020stability}, this leads to large chain complexes with many copies of the same column.  We imagine that there might be a way to extend our algorithms to work directly with the chain complex of a multi-critical bifiltration, and that this may be more efficient.

Finally, we would like to understand whether specialized algorithms for 3-parameter persistence modules can improve on the performance of known algorithms which work in greater generality.

\subsection*{Acknowledgements}
Many thanks to: William Wang, for valuable discussions and for writing the scripts used to obtain the timing data presented in this paper; Dave Turner, for parallelizing RIVET's code for minimizing a presentation; Bryn Keller, for extensive contributions to RIVET, including ones which made the parallelization possible, and also for pointing out several typos in an early version of this paper; Michael Kerber, for sharing the observation of \cref{Rem:Kerber}; Alex Tchernev, for sharing his insights about the size of Gr\"obner bases in the 2-parameter setting, as described in the proof of \cref{Prop:Tchernev}; Ulrich Bauer, for many enlightening discussions about 1-parameter persistence computation, which influenced this work; Roy Zhao, for helpful conversations about the clearing optimization and persistence computation more generally; and the anonymous reviewers, for several suggestions which led to significant improvements of this paper. Work on the paper began while the authors were postdoctoral fellows at the Institute for Mathematics and its Applications; we thank the IMA and its staff for their support.  This work was funded in part by NSF grant DMS-1606967.  In addition, Lesnick was supported in part by NIH grant T32MH065214 and an award from the J. Insley Blair Pyne Fund.


\bibliographystyle{abbrv}
{\small \bibliography{Bigraded_Betti_References} }

\end{document}